\documentclass[a4paper]{amsart}

\usepackage{amsrefs}

\usepackage{enumerate}

\title[Nonlinear Schr\"odinger equation]{Nonlinear Schr\"odinger equation with unbounded or vanishing potentials: solutions
concentrating on lower dimensional spheres}
\author{Denis Bonheure}
\address{
  D{\'e}partement de Math{\'e}matique\\
  Universit{\'e} libre de Bruxelles, CP 214\\
  Boulevard du Triomphe, 1050 Bruxelles, Belgium}
\email{denis.bonheure@ulb.ac.be}

\author{Jonathan Di Cosmo}
\address{
  D{\'e}partement de Math{\'e}matique\\
  Universit{\'e} cahtolique de Louvain\\
  Chemin du Cyclotron 2, 1348 Louvain-la-Neuve, Belgium}
\address{
  D{\'e}partement de Math{\'e}matique\\
  Universit{\'e} libre de Bruxelles, CP 214\\
  Boulevard du Triomphe, 1050 Bruxelles, Belgium}
\email{Jonathan.DiCosmo@uclouvain.be}
\thanks{Jonathan Di Cosmo is a research fellow of the Fonds de la Recherche Scientifique--FNRS}

\author{Jean Van Schaftingen}
\address{
  D{\'e}partement de Math{\'e}matique\\
  Universit{\'e} cahtolique de Louvain\\
  Chemin du Cyclotron 2, 1348 Louvain-la-Neuve, Belgium}
\email{Jean.VanSchaftingen@uclouvain.be}

\date{February 4, 2010}

\newcommand{\R}{\mathbb{R}}
\newcommand{\N}{\mathbb{N}}

\newcommand{\dist}{\mathop{\mathrm{dist}}}  		
\newcommand{\norm}[1]{\left\| #1 \right\|}		
\newcommand{\abs}[1]{\left| #1 \right|}			
\newcommand{\Dp}[2]{\frac{\partial #1}{\partial #2}}	

\newtheorem{proposition}{Proposition}[section]
\newtheorem{theorem}{Theorem}
\newtheorem{lemma}[proposition]{Lemma}
\newtheorem{corollary}[proposition]{Corollary}


\begin{document} 

\begin{abstract}
 We study positive bound states for the equation
\[
	- \varepsilon^2 \Delta u + V(x)u = K(x)f(u), \hspace{1cm}  x \in \R^N,
\]
where $\varepsilon > 0$ is a real parameter and $V$ and $K$ are radial positive potentials. We are especially interested in
solutions which concentrate on a $k$-dimensional sphere, $1 \leq k \leq N-1$, as $\varepsilon \rightarrow 0$. We adopt a
purely variational approach which allows us to consider broader classes of potentials than those treated in previous
works. For example, $V$ and $K$ might be singular at the origin or vanish superquadratically at infinity.
\end{abstract}

\maketitle

\section{Introduction}
We consider the nonlinear Schr\"odinger equation
\begin{align}\label{NLS}
i \hbar \Dp{\psi}{t} = -\frac{\hbar^2}{2m} \Delta \psi + W(x) \psi - \abs{\psi}^{p-1}\psi, \hspace{1cm} (t,x) \in \R
\times \R^N,
\end{align}
which appears for instance in nonlinear optics or condensed matter physics. A \textit{standing wave} solution of
\eqref{NLS} is a solution of the form
\begin{align*}
\psi(t,x) = e^{-iEt/\hbar} u(x),
\end{align*}
where $E$ is the energy of the wave. The function $\psi$ is a standing wave solution of \eqref{NLS} if and only if $u$
is a solution of the semilinear elliptic equation
\begin{align}\label{p1}
 - \varepsilon^2 \Delta u + V(x)u = \abs{u}^{p-1}u, \hspace{1cm}  x \in \R^N,
\end{align}
 where $\varepsilon^2 = \hbar^2/2m$ and $V(x) = W(x) - E$. It is a \textit{bound state} if $u \in H^1(\R^N)$. 
From a physical point of view, one expects to recover the laws of classical mechanics when $\hbar \rightarrow 0$. It is
thus interesting to study the behaviour of the solutions of \eqref{p1} as $\varepsilon$ tends to $0$.
The bound states of \eqref{p1} with $\varepsilon$ small are referred to as \textit{semiclassical states}. 

It is well known that problem \eqref{p1} possesses solutions which exhibit concentration phenomena as $\varepsilon
\rightarrow 0$. More precisely, these solutions converge uniformly to $0$ outside some concentration set, while
remaining uniformly positive in the concentration set. This concentration set can be either a point, a finite set of
points or a manifold. 

The solutions concentrating around one or several isolated points have been intensively studied (see for example
\cites{AM,BVS} and their bibliographies). 

On the other hand, one can ask if there exist solutions of \eqref{p1} concentrating on a higher dimensional set. This problem has been solved for some specific higher dimensional sets. Solutions concentrating on curves have been found recently in \cite{MMM},
see also \cite{DKW} for the case $N=2$ and \citelist{\cite{MaMa}\cite{Mal}} for a Neumann singularly perturbed problem. Here we shall
restrict ourselves to the problem of solutions concentrating around spheres. In several recent papers
\citelist{\cite{BeDA}\cite{BaDA}\cite{AMN}\cite{AR}\cite{ByW}\cite{BaP}}, solutions concentrating on $(N-1)$-dimensional spheres have been found. In \cite{MP},
solutions concentrating on $(N-2)$-dimensional spheres are investigated. 

We focus on solutions
concentrating around a $k$-dimensional sphere in $\R^N$, $1 \leq k \leq N-1$. The existence of such solutions has been
discussed in remarks in \cite{AM2}, \cite{AM}, \cite{ByW}.
Particular problems arise in the \textit{critical frequency} case, namely when $\inf_{\R^N} V = 0$. These problems have been
tackled in \cite{AR} and \cite{ByW}.

\begin{theorem}[Ambrosetti-Ruiz \cite{AR}]\label{ThAR}
Assume that $p > 1$, that $V \in C^1(\R^N)$ is a positive bounded radially symmetric potential, that $\nabla V$ is bounded and that 
\[
 \liminf_{\abs{x} \to \infty} V(x)\abs{x}^2 > 0.
\]
If there exists $r^*$ such that the function $\mathcal{M} : (0, \infty) \to \R$ defined for $r > 0$ by
\begin{align}\label{def:Mr}
 \mathcal{M}(r) := r^{N-1} \left[V(r)\right]^{\frac{p+1}{p-1}-\frac{1}{2}}
\end{align}
has an isolated local maximum or minimum at $r = r^*$, then, for $\varepsilon > 0$ small enough, equation \eqref{p1} has a
positive radially symmetric solution $u_{\varepsilon} \in H^1(\R^N)$ that concentrates at the sphere $\abs{x} = r^*$.
\end{theorem}

The problem in \cite{ByW} is rather different. The potential $V$ vanishes and the solutions concentrate around zeroes of
$V$. The asymptotic behaviour depends on the shape of $V$ around $0$.

Theorem \ref{ThAR} relies on a Lyapunov-Schmidt reduction method. The aim of this note is to examine possible
improvements in the previous results that can be obtained by using the penalization method, a variational method
originally due to Del Pino and Felmer \cite{DF1} and adapted to our framework in the papers \citelist{\cite{BVS}\cite{MVS}}. This method
permits us to treat superquadratically decaying potentials, or even compactly supported potentials.
 
Our results include the following simple particular case.
\begin{theorem}\label{theo:special}
 Let $N \geq 3$, $p>\frac{N}{N-2}$ and $V \in C(\R^N \backslash \left\{ 0 \right\},\R^+)$ be a
radial potential. If there exists $r^* > 0$ such that the function $\mathcal{M}(r)$ defined by
\eqref{def:Mr} has an isolated local minimum at $r = r^*$ such that $\mathcal{M}(r^*)>0$, then for
$\varepsilon$ small enough, the equation \eqref{p1} has a positive radially symmetric solution
$u_{\varepsilon}$ that concentrates on the sphere of radius $r^*$.
\end{theorem}

If $N \geq 5$, one has also that $u_\varepsilon \in L^2(\R^N)$ (see Corollary~\ref{cor:L2}).

In contrast with Theorem \ref{ThAR}, we do not require any boundedness assumption on $V$ or its derivatives, and we
treat potentials $V$ which are singular at the origin or vanish superquadratically at infinity.

Theorem \ref{theo:special} is a particular case of Theorem \ref{Th:main} below, which deals with a nonlinearity which is
neither necessarily homogeneous nor autonomous, see equation \eqref{p1f} below. Furthermore, we will find solutions
concentrating on a $k$-dimensional sphere, $1 \leq k \leq N-1$. In this case, the critical exponent to be taken into
consideration is $p_k = \frac{N-k+2}{N-k-2}$ if $N-k \geq 3$, $p_k = \infty$ if $N-k = 1,2$. We also obtain results for $N=2$ with a little more care, see Section \ref{rem:N=2}.

Let us point out that if $V$ is compactly supported and $p \leq \frac{N}{N-2}$, then equation \eqref{p1} has no positive
solution in the neighborhood of infinity, see the discussion in \cite{MVS}.

Assuming that the potential $V$ is cylindrically symmetric, we can reduce \eqref{p1} to a problem in $\R^{N-k}$. The
single-peaked solutions of this problem can then be extended to $\R^N$ by symmetry. In this way, we obtain a solution of
\eqref{p1} concentrating around a $k$-dimensional sphere. Observe that since the reduced problem is in $\R^{N-k}$, the
critical exponent to be considered is the one in dimension $N-k$. This allows for example to treat critical problems by
looking for cylindrically symmetric (non necessarily radial) solutions.

\section{Assumptions and main result}
We shall study the equation with a more general nonlinearity
\begin{align}\label{p1f}
 - \varepsilon^2 \Delta u + V(x)u = K(x)f(u), \hspace{1cm}  x \in \R^N.
\end{align}

Let $k$ be a fixed integer such that $1 \leq k \leq N-1$. This number $k$ is the dimension of the
sphere on which we want to construct concentrating solutions. Let us choose any
$(N-k-1)$-dimensional linear subspace $\mathcal{H} \subset \R^N$. We denote by $\mathcal{H}^{\bot}$
the orthogonal complement of $\mathcal{H}$. If $x \in \R^N$, we will write $x = (x',x'')$ with $x'
\in \mathcal{H}$ and $x'' \in \mathcal{H}^{\bot}$.

\subsection{The potentials}
We consider a nonnegative potential $V \in C(\R^N \backslash \left\{ 0 \right\})$ and a nonnegative competing function
$K \in C(\R^N \backslash \left\{ 0 \right\})$, $K \not\equiv 0$. 
We assume that for every $R \in \mathbf{O}(N)$ such that $R(\mathcal{H})=\mathcal{H}$, we have $V \circ R = V$ and $K \circ R = K$. This will be the case if for example $V$ and $K$ are radial functions.

\subsection{The nonlinearity}
We make classical assumptions on $f$ that lead to a good minimax characterization of the infimum on the Nehari
manifold. Namely, we assume that $f : \R^+ \rightarrow \R^+$ is continuous and that
\begin{itemize}
 \item[$(f_1)$] there exists $q>1$ such that $f(s) = O(s^q)$ as $s \rightarrow 0^+$,
 \item[$(f_2)$] there exists $p>1$ such that $\frac{1}{p+1} > \frac{1}{2} - \frac{1}{N-k}$ and $f(s) = O(s^p)$ as $s
\rightarrow \infty$,
 \item[$(f_3)$] there exists $2 < \theta \leq p+1$ such that
 \begin{align*}
  0 < \theta F(s) \leq f(s)s \hspace{1cm} \text{for}\ s>0,
 \end{align*}
 where $F(s) := \int_0^s f(\sigma) d\sigma$,
 \item[$(f_4)$] the function
 \begin{align*}
 s \mapsto \frac{f(s)}{s}
 \end{align*}
 is nondecreasing.
\end{itemize}
Notice that $(f_2)$ is nothing but the subcriticality condition in dimension $N-k$.

\subsection{The growth conditions}
Following \citelist{\cite{BVS}\cite{MVS}} we impose one of the three sets of growth conditions at infinity :
\begin{itemize}
\item[$(\mathcal{G}_{\infty}^1)$] there exists $\sigma < (N-2)q-N$ such that
\[
  \limsup_{\abs{x} \to \infty} \frac{K(x)}{\abs{x}^\sigma} < \infty;
\]
\item[$(\mathcal{G}_{\infty}^2)$] there exists $\sigma \in \R$ such that
 \begin{align*}
  \liminf_{\abs{x} \rightarrow \infty} V(x) \abs{x}^{2} &> 0\ &&\text{ and } &  \limsup_{\abs{x} \to \infty} \frac{K(x)}{\abs{x}^\sigma} &<\infty;
 \end{align*}
\item[$(\mathcal{G}_{\infty}^3)$] there exist $\alpha < 2$ and $\sigma \in \R$ such that
 \begin{align*}
  \liminf_{\abs{x} \rightarrow \infty} V(x) \abs{x}^{\alpha} &> 0\ & & \text{ and }\ &\limsup_{\abs{x} \to \infty} \frac{K(x)}{\exp ( \sigma
\abs{x}^{\frac{2-\alpha}{2}} )} &< \infty.
 \end{align*}
\end{itemize}
Note that in comparison with \cite{BVS}, in $(\mathcal{G}_{\infty}^2)$ and $(\mathcal{G}_{\infty}^3)$, $V$ might vanish somewhere.
We also impose one of the three sets of growth conditions at the origin, which mirror those at infinity :
\begin{itemize}
\item[$(\mathcal{G}_{0}^1)$] there exists $\tau > -2$, such that
\[
   \limsup_{\abs{x} \to 0} \frac{K(x)}{\abs{x}^{\tau}} < \infty,
\]
\item[$(\mathcal{G}_{0}^2)$] there exists $\tau \in \R$ such that
 \begin{align*}
  \liminf_{\abs{x} \rightarrow 0} V(x) \abs{x}^{2} &> 0\ &&\text{ and }\ &\limsup_{\abs{x} \to 0} \frac{K(x)}{\abs{x}^{\tau}} < \infty;
 \end{align*}
\item[$(\mathcal{G}_{0}^3)$]
 there exist $\gamma > 2$ and $\tau \in \R$ such that
 \begin{align*}
  \liminf_{\abs{x} \rightarrow 0} V(x) \abs{x}^{\gamma} &> 0\ &&\text{ and }\ & \limsup_{\abs{x} \to 0} \frac{K(x)}{\exp ( \tau
\abs{x}^{-\frac{\gamma-2}{2}} )} &< \infty.
 \end{align*}
\end{itemize}

By Kelvin transform, there is a duality between the conditions at the origin and the conditions at infinity, at least in the case where $f(t)=t^p$. If one defines $\Hat{u}$ to be the Kelvin transform of $u$, i.e., 
\[
 \Hat{u}(x)=\frac{1}{\abs{x}^{N-2}} u\Bigl(\frac{x}{\abs{x}^2} \Bigr)
\]
and the transformed potentials
\begin{align*}
 \Hat{V}(x)=\frac{1}{\abs{x}^4}V\Bigl(\frac{x}{\abs{x}^2} \Bigr)
\end{align*}
and
\begin{align*}
 \Hat{K}(x)=\frac{1}{\abs{x}^{N+2-p(N-2)}}K\Bigl(\frac{x}{\abs{x}^2} \Bigr),
\end{align*}
the function $u_\varepsilon$ solves \eqref{p1f} if and only if $\Hat{u}_\varepsilon$ solves the same problem with $\Hat{V}$ and $\Hat{K}$ in place of $V$ and $K$.
One sees that $V$, $K$ satisfy $(\mathcal{G}_{0}^i)$ if and only if $\Hat{V}$ and $\Hat{K}$  satisfy $(\mathcal{G}_{\infty}^i)$.

 The problem at the origin is in a sense in duality with the one at infinity. Whereas a slow decay of $V$ at infinity
does allow a lot of freedom for $K$, a strong singularity at the origin allows for very singular $K$'s too. The critical
threshold growth is $1/|x|^2$ both at the origin and at infinity. This can be made clearer if we observe that the
optimal barrier functions at the origin are the optimal one at infinity mapped by Kelvin transform.

\subsection{The auxiliary potential}\label{sect:Lambda}
Before we can state our last assumption, we need a few preliminaries. Let $a,b > 0$. The equation
\begin{align}\label{plim}
	 -\Delta u + au = bf(u) \hspace{1cm} \text{in} \ \R^{N-k}
\end{align}
is called the \textit{limit equation} associated with \eqref{p1f}. The weak solutions of \eqref{plim} are critical
points of the functional $\mathcal{I}_{a,b} : H^1(\R^{N-k}) \rightarrow \R$ defined by
\begin{align}\label{defIab}
 \mathcal{I}_{a,b}(u) := \frac{1}{2} \int_{\R^{N-k}} \left(  \abs{\nabla u}^2 + a u^2 \right)\: dx - b \int_{\R^{N-k}}
F(u)\: dx.
\end{align}
Any nontrivial critical point  $u \in H^1(\R^{N-k})$ of $\mathcal{I}_{a,b}$, belongs to
the Nehari manifold
\begin{align*}
 \mathcal{N}_{a,b} := \left\{ u \in H^1(\R^{N-k}) \ \vert\ u \not\equiv 0\ \text{and}\ \langle \mathcal{I}_{a,b}'(u),u
\rangle = 0 \right\}.
\end{align*}
A solution $u \in H^1(\R^{N-k})$ is a \textit{least-energy solution} of \eqref{plim} if
\begin{align*}
 \mathcal{I}_{a,b}(u) = \inf_{v \in \mathcal{N}_{a,b}} \mathcal{I}_{a,b}(v).
\end{align*}
The \textit{ground-energy function} is defined by
\begin{align*}
 \mathcal{E} : \R^+\times \R^+ \rightarrow \R^+ : (a,b) \mapsto \mathcal{E}(a,b) := \inf_{u \in \mathcal{N}_{a,b}}
\mathcal{I}_{a,b}(u),
\end{align*}
and the \textit{auxiliary potential} $\mathcal{M} :\ \R^{N} \rightarrow (0,+\infty]$ by
\[
 (x',x'') \mapsto \mathcal{M}(x',x'') := \left\{ \begin{array}{ll} \abs{x''}^{k} \mathcal{E}\left(V(x),K(x)\right) &
\text{if}\ K(x) > 0, \\
 +\infty & \text{if}\ K(x)=0.
 \end{array} 
 \right.
\]
The following lemma states some properties of the ground-energy function, see \cite{BVS}*{Lemma 3}.
\begin{lemma}
 Assume $f : \R^+ \rightarrow \R^+$ is a continuous function that fulfills assumptions ($f_1$)-($f_4$). Then, for every
$(a,b) \in \R^+_0 \times \R^+_0$, $\mathcal{E}(a,b)$ is a critical value of $\mathcal{I}_{a,b}$ and we have
\begin{align*}
 \mathcal{E}(a,b) = \inf_{\substack{u \in H^1(\R^N) \\ u \neq 0}} \max_{t \geq 0} \; \mathcal{I}_{a,b}(tu).
\end{align*}
If $u \in \mathcal{N}_{a,b}$ and $\mathcal{E}(a,b) = \mathcal{I}_{a,b}(u)$, then $u \in C^1(\R^N)$ and up to a
translation, $u$ is a radial function such that $\nabla u(x) \cdot x < 0$ for every $x \in \R^N\setminus \{0\}$.
Moreover, the following properties hold:
\begin{itemize}
 \item[(i)] $\mathcal{E}$ is continuous in $\R^+_0 \times \R^+_0$;
 \item[(ii)] for every $b^* \in \R^+_0$, $a \to \mathcal{E}(a,b^*)$ is strictly increasing;
 \item[(iii)] for every $a^* \in \R^+_0$, $b \to \mathcal{E}(a^*,b)$ is strictly decreasing;
 \item[(iv)] for every $\lambda > 0$, $\mathcal{E}(\lambda a, \lambda b) = \lambda^{1-N/2} \mathcal{E}(a,b)$;
 \item[(v)] if $f(u) = u^p$ with $\frac{1}{2} - \frac{1}{N-k} < \frac{1}{p+1} < \frac{1}{2}$, then
\begin{align*}
 \mathcal{E}(a,b) = \mathcal{E}(1,1) a^{\frac{p+1}{p-1}-\frac{N}{2}} b^{-\frac{2}{p-1}}.
\end{align*}
\end{itemize}
\end{lemma}

If $f(u) = u^p$ with $\frac{1}{2} - \frac{1}{N-k} < \frac{1}{p+1} < \frac{1}{2}$, the last property of the preceding
lemma implies the following explicit form of the auxiliary potential:
\[
\label{defM2}
 \mathcal{M}(x',x'') = \mathcal{E}(1,1) \abs{x''}^k \left[V(x)\right]^{\frac{p+1}{p-1} - \frac{N-k}{2}}
\left[K(x)\right]^{\frac{-2}{p-1}}.
\]

Due to the symmetry that we shall impose on the solution (see \eqref{def:spaceH}), the
concentration can only occur in the space $\mathcal{H}^{\bot}$. We assume that there exists a smooth
bounded open set $\Lambda \subset \R^{N}$ such that  
\begin{equation}
\label{eqLambdacapH}
\Bar{\Lambda} \cap \mathcal{H} = \emptyset, \ \Lambda \cap \mathcal{H}^\perp \neq \emptyset,
\end{equation}
for every $R \in \mathbf{O}(N)$ such that $R(\mathcal{H})=\mathcal{H}$, 
\begin{align}
\label{symLambda}
  R(\Lambda)=\Lambda
\end{align}
and
\begin{align}\label{hyp:Lambda}
 0 < \inf_{\Lambda \cap \mathcal{H}^\perp}\mathcal{M} < \inf_{\partial \Lambda \cap \mathcal{H}^\perp}\mathcal{M}.
\end{align}
In the case where $k=N-2$, we shall need the condition
\begin{align}\label{def:Lambda}
 \inf_{\Lambda \cap \mathcal{H}^\perp}\mathcal{M} < 2 \inf_{\Lambda} \mathcal{M}.
\end{align}
By continuity of $\mathcal{M}$ in $\Lambda$, this condition is not restrictive.
Similarly, we can also assume that $V > 0$ on $\overline{\Lambda}$ and that $\mathcal{M}$ is continuous on $\overline{\Lambda}$.

\bigskip

Our main result is the following theorem.
\begin{theorem}\label{Th:main}
Let $N \geq 2$, $V, K \in C(\R^N \backslash \left\{ 0 \right\},\R^+)$ satisfy one set $(\mathcal{G}_{0}^i)$ of growth conditions
at the origin and one set $(\mathcal{G}_{\infty}^j)$ of growth conditions at infinity, and $f$ satisfy assumptions $(f_1)$-$(f_4)$. 
Assume there exists an open
bounded set $\Lambda \subset \R^{N}$ such that \eqref{eqLambdacapH}, \eqref{symLambda}, \eqref{hyp:Lambda} and, if $k=N-2$, \eqref{def:Lambda} hold. Then there exists $\varepsilon_0 > 0$ such that for every $0 < \varepsilon <
\varepsilon_0$, problem \eqref{p1} has at least one positive solution $u_{\varepsilon}$. Moreover, for every $0 < \varepsilon <
\varepsilon_0$, there exists $x_{\varepsilon} \in \Lambda \cap \mathcal{H}^\perp$ such that $u_{\varepsilon}$ attains its maximum at $x_{\varepsilon}$,
\begin{align*}
 \liminf_{\varepsilon \to 0} u_{\varepsilon}(x_{\varepsilon}) &> 0, \\
 \lim_{\varepsilon \to 0} \mathcal{M} (x_{\varepsilon}) &= \inf_{\Lambda \cap \mathcal{H}^\perp} \mathcal{M},
\end{align*}
and there exist $C>0$ and $\lambda > 0$ such that
\begin{align*}
 u_{\varepsilon}(x) &\leq C \exp{\left( -\frac{\lambda}{\varepsilon} \frac{d(x,S^k_{\varepsilon})}{1+d(x,S^k_{\varepsilon})}\right) }
\left( 1+\abs{x}^2 \right)^{\frac{-(N-2)}{2}}, &  \forall x \in \R^N,
\end{align*}
where $S^k_{\varepsilon}$ is the $k$-sphere centered at the origin and of radius $\abs{x''_{\varepsilon}}$.
\end{theorem}

In the special case where $x_0 \in \Lambda \cap \mathcal{H}^\perp$ is the unique minimizer of $\mathcal{M}$ on $\Lambda \cap \mathcal{H}^\perp$, then $x_\varepsilon \to x_0$, and the solution concentrates around a
$k$--dimensional sphere of radius $\abs{x_0}$ centered at the origin.

One should note that the theorem is valid in dimension 2, but the solutions that are obtained do not decay at infinity in general.

Theorem \ref{theo:special} follows from Theorem~\ref{Th:main} by taking $K \equiv 1$, $f(u) = u^p$ and $k=N-1$. Indeed, we notice that the growth condition $(\mathcal{G}_{0}^1)$ is always satisfied whereas the condition $(\mathcal{G}_{\infty}^1)$ holds if
and only if $(N-2)p - N > 0$, i.e. $p>\frac{N}{N-2}$.

The sequel of the paper is devoted to the proof of Theorem~\ref{Th:main}. In Section \ref{sec:Penal}, we introduce a penalized
problem and prove that it has a least energy solution. In Section \ref{sec:Asympt}, we study the asymptotics of
this solution and in Section \ref{sec:barrier}, we obtain decay estimates of the solution and show that it also solves
the original problem. In all these sections, we assume that $N\geq 3$. The modifications for the case $N=2$ will be addressed
in Section \ref{rem:N=2}.

\section{The penalization scheme}\label{sec:Penal}

We assume that $N\geq 3$. Let $\mathcal{D}(\R^N)$ be the set of compactly supported smooth functions. The homogeneous
Sobolev space $\mathcal{D}^{1,2}(\R^N)$ is the closure of the set of compactly supported smooth functions $\mathcal{D}(\R^N)$ with respect to the norm 
\[
\left( \int_{\R^N}
\abs{\nabla u}^2\: dx \right)^{\frac{1}{2}}. 
\]
Thanks to Sobolev inequality, we have $\mathcal{D}^{1,2}(\R^N) \subset
L^{2^*}(\R^N)$. Let us also recall Hardy's inequality in $\R^N$. One has
\begin{align*}
\left( \frac{N-2}{2} \right)^2 \int_{\R^N} \frac{\abs{u(x)}^2}{\abs{x}^2}\: dx \leq \int_{\R^N} \abs{\nabla u}^2,
\end{align*}
for all $u \in \mathcal{D}^{1,2}(\R^N)$.


Following \cite{MVS}, we define the \emph{penalization potential} $H
: \R^N \to \R$ by
\begin{align*}
 H(x) := \frac{\kappa}{\abs{x}^2 \Bigl( \bigl(\log \abs{x} \bigr)^2+1 \Bigr)^{\frac{1+\beta}{2}}}
\end{align*}
where $\beta > 0$ and $0 < \kappa < (\frac{N-2}{2})^2$.
Notice that for all $x \in \R^N$, we have
\[
 H(x) \leq \frac{\kappa}{\abs{x}^2}.
\]
By Hardy's inequality, we deduce that the quadratic form associated to $- \Delta - H$ is positive, i.e.
\begin{align}\label{positivity}
 \int_{\R^N} \left( \abs{\nabla u}^2 - H u^2 \right) \geq \biggl(  \Bigl( \frac{N-2}{2} \Bigr)^2 - \kappa \biggr)
\int_{\R^N} \frac{\abs{u(x)}^2}{\abs{x}^2}\: dx \geq 0,
\end{align}
for all $u \in \mathcal{D}^{1,2}(\R^N)$. 

This inequality implies the following comparison principle.
\begin{proposition}\label{Th:comp}
Let $\Omega \subset \R^N \setminus\{0\}$ be a smooth domain. Let $v,w \in H^1_{\textnormal{loc}}(\Omega)\cap
C(\overline{\Omega})$ be such
that $\nabla (w-v)_- \in L^2(\Omega)$, $(w-v)_-/\abs{x} \in L^2(\Omega)$ and
\begin{align}\label{comp}
 - \Delta w - H w \geq - \Delta v - H v, \hspace{1cm} \forall x \in \Omega.
\end{align}
If $\partial\Omega \neq \emptyset$, assume also that $w \geq v$ on $\partial\Omega$. Then $w \geq v$ in $\Omega$.
\end{proposition}
\begin{proof}
 It suffices to multiply the inequality \eqref{comp} by $(w-v)_-$, integrate by parts and use \eqref{positivity}.
\end{proof}


Fix $\mu \in (0,1)$. We define the penalized nonlinearity $g_{\varepsilon}: \R^N \times \R^+ \rightarrow \R$ by
\[
  g_{\varepsilon}(x,s) := \chi_{\Lambda}(x) K(x) f(s) + \left( 1-\chi_{\Lambda}(x) \right) \min\left\lbrace  \left(
\varepsilon^2 H(x)+\mu V(x)\right) s, K(x)f(s) \right\rbrace.
\]
Let $G_{\varepsilon}(x,s) := \int_0^s g_{\varepsilon}(x,\sigma) d\sigma$. One can check that $g_{\varepsilon}$ is a
Carath\'eodory
function with the following properties :
\begin{itemize}
 \item[($g_1$)] $g_{\varepsilon}(x,s) = o(s),\ s \rightarrow 0^+$, uniformly in compact subsets of $\R^N$.
 \item[($g_2$)] there exists $p>1$ such that $\frac{1}{p+1} > \frac{1}{2} - \frac{1}{N-k}$ and 
 \begin{align*}
  \lim_{s \rightarrow \infty} \frac{g_{\varepsilon}(x,s)}{s^p} = 0,
 \end{align*}
 \item[($g_3$)] there exists $2 < \theta \leq p+1$ such that
 \begin{align*}
 \begin{array}{ll}
  0 < \theta G_{\varepsilon}(x,s) \leq g_{\varepsilon}(x,s)s &\forall x \in \Lambda,\ \forall s>0, \\
  0 < 2 G_{\varepsilon}(x,s) \leq g_{\varepsilon}(x,s)s \leq \left( \varepsilon^2 H(x)+\mu V(x)\right) s^2 &\forall x \notin
\Lambda,\ \forall s>0,
 \end{array}
 \end{align*}
 \item[($g_4$)] the function
 \begin{align*}
 s \mapsto \frac{g_{\varepsilon}(x,s)}{s}
 \end{align*}
 is nondecreasing for all $x \in \R^N$.
\end{itemize}

\bigskip

We look for a positive solution of the penalized equation
\begin{align}\label{p2}\tag{$\mathcal{P}_{\varepsilon}$}
	 - \varepsilon^2 \Delta u + V(x)u = g_{\varepsilon}(x,u) \hspace{1cm} \textnormal{in} \ \R^N
\end{align}
in the Hilbert space
\begin{align*}
 H^1_{V}(\R^N) := \left\{ u \in \mathcal{D}^{1,2}(\R^N)\ \vert\ \int_{\R^N} V u^2 < \infty \right\}
\end{align*}
endowed with the norm
\begin{equation}
\label{eqNorm}
 \norm{u}_{\varepsilon}^2 := \int_{\R^N} \left( \varepsilon^2 \abs{\nabla u}^2 + V u^2 \right).
\end{equation}
We will search for a solution of \eqref{p2} in the closed subspace
\begin{multline}\label{def:spaceH}
 H^1_{V,\mathcal{H}}(\R^N) := \bigl\{ u \in H^1_{V}(\R^N)\ \vert\ \forall R \in \mathbf{O}(N) \ \text{s.t.}\
R({\mathcal{H}}) = {\mathcal{H}},
 u \circ R = u \bigr\}.
\end{multline}

\bigbreak

Define $J_{\varepsilon} : H^1_{V,\mathcal{H}}(\R^N) \rightarrow \R$ by
\begin{align*}
  J_{\varepsilon}(u) := \frac{1}{2} \int_{\R^N} \left( \varepsilon^2
\abs{\nabla u(x)}^2 + V(x)\abs{u(x)}^2 \right)\: dx - \int_{\R^N} G_{\varepsilon}(x,u(x))\: dx.
\end{align*}
The functional $J_{\varepsilon}$ is well defined and of class $C^1(H^1_{V,\mathcal{H}}(\R^N),\R)$.
By the principle of symmetric criticality \cite{Palais}, critical points are weak solutions of \eqref{p2}. Furthermore,
$J_{\varepsilon}$ has the mountain pass
geometry. It remains to show that $J_{\varepsilon}$ satisfies the Palais-Smale condition. The proof below is inspired from
\cite{BVS}. Recall that a sequence $(u_n)_n \subset H^1_{V,\mathcal{H}}(\R^N)$ is a Palais-Smale sequence for
$J_{\varepsilon}$ if
\begin{align*}
 J_{\varepsilon}(u_n) \leq C \hspace{1cm} \textnormal{and}\hspace{1cm} J'_{\varepsilon}(u_n) \rightarrow 0,\ n \rightarrow
\infty.
\end{align*}

\begin{proposition}\label{prop:PS}
 For $\varepsilon$ sufficiently small, every Palais-Smale sequence for $J_{\varepsilon}$ contains a convergent subsequence.
\end{proposition}
\begin{proof}
 Let $(u_n)_n \subset H^1_{V,\mathcal{H}}(\R^N)$ be a Palais-Smale sequence for $J_{\varepsilon}$. It is standard to check,
using $(g_3)$, that for $\varepsilon$ sufficiently small, the sequence $(u_n)_n$ is bounded in $H^1_{V,\mathcal{H}}(\R^N)$.
We infer that, up to a subsequence, $u_n \rightharpoonup u$ in $H^1_{V,\mathcal{H}}(\R^N)$. 

For $\lambda \in \R^+$, set $A_{\lambda} := \overline{B(0,e^\lambda)} \setminus B(0,e^{-\lambda})$. 
Note that 
\[
 H(x)\leq \frac{\kappa}{\abs{x}^2 \abs{\log \abs{x}}^{1+\beta}}.
\]
By Hardy's inequality, we have for $\lambda \geq 0$, 
\begin{align*}
 \int_{\R^N\setminus A_{\lambda}} H u_n^2 &\leq \frac{\kappa}{\lambda^{1+\beta}}
\int_{\R^N} \frac{\abs{u_n(x)}^2}{\abs{x}^2}\: dx
 \leq \frac{\kappa}{\lambda^{1+\beta}} \left( \frac{2}{N-2} \right)^2
\int_{\R^N} \abs{\nabla u_n}^2.
\end{align*}
Since $(u_n)_n$ is bounded in $H^1_{V,\mathcal{H}}(\R^N)$, for every $\delta > 0$, there exists $\Bar{\lambda} \geq 0$
such that 
\begin{align}\label{ps:claim1}
 \limsup_{n \rightarrow \infty} \int_{\R^N\setminus A_{\Bar{\lambda}}} H u_n^2 < \delta.
\end{align}

Now we claim that for all $\delta > 0$, there exists $\Tilde{\lambda} > 0$ such that
\begin{align}\label{ps:claim2}
  \limsup_{n \rightarrow \infty} \int_{\R^N\setminus A_{\Tilde{\lambda}}} V u_n^2 < \delta.
\end{align}
We only sketch the proof, since the arguments are similar to those in \cite{BVS}*{Lemma 6}. 
Since $\Bar{\Lambda} \subset \R^N \setminus \{0\}$ is compact, there exists $\lambda_0 \geq 0$ such that 
\[
 \Bar{\Lambda} \subset A_{\lambda_0}.
\]
Let $\zeta \in C^{\infty}(\R)$ be such that $0 \leq \zeta \leq 1$ and
\begin{align*}
  \zeta(s) = \left\{ \begin{array}{ll} 0 & \textnormal{if}\ \abs{s} \leq \frac{1}{2}, \\
  1 & \textnormal{if}\ \abs{s} \geq 1. 
  \end{array} 
  \right.
\end{align*}
Define a cut-off function $\eta_{\lambda} \in C^{\infty}(\R^N,\R)$ by
 \begin{align*}
  \eta_{\lambda}(x) := \zeta\left( \frac{\log \abs{x}}{\lambda} \right).
 \end{align*}
Since $\langle J'_{\varepsilon}(u_n),\eta_{\lambda}u_n\rangle = o(1)$ as $n \to \infty$, we deduce that
 \begin{multline}\label{ps3}
	\int_{\R^N} \left( \varepsilon^2 \abs{\nabla u_n}^2 + V u_n^2 \right)\eta_{\lambda} = \int_{\R^N} g_{\varepsilon}(x, u_n(x))
u_n(x) \eta_{\lambda}(x)\: dx  \\
- \varepsilon^2 \int_{\R^N} u_n \nabla u_n \cdot \nabla \eta_{\lambda} + o(1),
 \end{multline}
as $n \to \infty$.
If $\lambda \geq 2\lambda_0$, $\eta_\lambda=0$ on $\Lambda$ and it follows from $(g_3)$ that
 \begin{align}\label{ps4}
  \int_{\R^N} g_{\varepsilon}(x,u_n(x)) u_n(x) \eta_{\lambda}(x)\: dx \leq \int_{\R^N} \left( \varepsilon^2 H + \mu V \right)
u_n^2 \eta_{\lambda}.
 \end{align}
On the other hand, using Hardy's inequality, we can show as in \cite{BVS} that
 \begin{align}\label{ps5}
  \abs{\int_{\R^N} u_n \nabla u_n \cdot \nabla \eta_{\lambda}} \leq \frac{C}{\lambda} \norm{u_n}_{\varepsilon}^2.
 \end{align}
Combining \eqref{ps3}, \eqref{ps4} and \eqref{ps5}, we get, for every $\lambda \geq 2\lambda_0$, 
\[
\begin{split}
  \int_{\R^N\setminus A_{\lambda}} \left( \varepsilon^2 \abs{\nabla u_n}^2 + (1-\mu) V u_n^2 \right)\eta_{\lambda}
&\leq
  \int_{\R^N} \left( \varepsilon^2 \abs{\nabla u_n}^2 + (1-\mu) V u_n^2 \right)\eta_{\lambda} \\
&\leq \int_{\R^N} \varepsilon^2 H
u_n^2 \eta_{\lambda} + \frac{C}{\lambda} \norm{u_n}_{\varepsilon}^2 + o(1).
\end{split}
 \]
By \eqref{ps:claim1}, for $\lambda$ large enough,
\begin{align*}
 \limsup_{n \rightarrow \infty} \int_{\R^N} H u_n^2 \eta_{\lambda} \leq \limsup_{n \rightarrow \infty}
\int_{\R^N\setminus A_{\Bar{\lambda}}} H u_n^2 < \frac{\delta}{2};
\end{align*}
the claim follows.


\medbreak \textit{Conclusion.} We can write
\begin{multline}\label{ps10}
 \norm{u_n - u}_{\varepsilon}^2 = J'_{\varepsilon}(u_n)(u_n-u) - J'_{\varepsilon}(u)(u_n-u)  \\
+ \int_{\R^N} \left(
g_{\varepsilon}(x,u_n(x))-g_{\varepsilon}(x,u(x)) \right) \left(u_n(x)-u(x)\right)\: dx.
\end{multline}
We notice that the first two terms in the right-hand side tend to $0$ as $n \rightarrow \infty$. Fix $\delta > 0$ and
let $\lambda > 0$ be such that \eqref{ps:claim1} and \eqref{ps:claim2} hold. We evaluate the integral in the third
term of \eqref{ps10} separately on $\Lambda$, $A_{\lambda}\setminus \Lambda$ and $\R^N\setminus A_{\lambda}$, where
$\lambda =\max\{\Tilde{\lambda}, \Bar{\lambda}\}$.

By $(g_2)$, one has $\abs{g_{\varepsilon}(x,u_n(x))} \leq C \abs{u_n(x)}^p$. By Rellich Theorem, the embedding
$H^1_{V,\mathcal{H}}(\Lambda) \hookrightarrow L^q(\Lambda)$ is compact for all $q>1$ such that 
$\frac{1}{q} > \frac{1}{2} - \frac{1}{N-k}$. We can thus assume that $u_n \rightarrow u$ in $L^{p+1}(\Lambda)$. We
deduce that $g_{\varepsilon}(x,u_n) \rightarrow g_{\varepsilon}(x,u)$ in $L^q(\Lambda)$ as $n\to \infty$, where $q :=
\frac{p+1}{p}$. We conclude from H\"older inequality that
\begin{align*}
 \int_{\Lambda} \left( g_{\varepsilon}(x,u_n(x))-g_{\varepsilon}(x,u(x)) \right) \left(u_n(x)-u(x)\right)\: dx \rightarrow 0,
\hspace{0.2cm} \text{as} \ \ n \rightarrow \infty.
\end{align*}

By $(g_3)$, one has $\abs{g_{\varepsilon}(x,u_n(x))} \leq \left( \varepsilon^2 H(x)+\mu V(x) \right) \abs{u_n(x)}$ for $x \in
A_{\lambda}\setminus \Lambda$. By Rellich Theorem, we can assume that $u_n \rightarrow u$ in
$L^{2}(A_{\lambda}\setminus \Lambda)$. We deduce that $g_{\varepsilon}(x,u_n) \rightarrow g_{\varepsilon}(x,u)$ in
$L^2(A_{\lambda}\setminus \Lambda)$ as $n\to \infty$. We conclude as above that
\begin{align*}
 \int_{A_{\lambda}\setminus \Lambda} \left( g_{\varepsilon}(x,u_n(x))-g_{\varepsilon}(x,u(x)) \right)
\left(u_n(x)-u(x)\right)\: dx \rightarrow 0, \hspace{0.2cm} \text{as} \ \ n \rightarrow \infty.
\end{align*}

Finally, using $(g_3)$, \eqref{ps:claim1} and \eqref{ps:claim2}, we obtain
\begin{align*}
 &\limsup_{n \rightarrow \infty} \int_{\R^N\setminus A_{\lambda}} \abs{g_{\varepsilon}(x,u_n(x))-g_{\varepsilon}(x,u(x))}
\abs{u_n(x)-u(x)}\: dx \\ &\qquad \leq 2 \limsup_{n \rightarrow \infty} \int_{\R^N\setminus A_{\lambda}} \left(
\abs{g_{\varepsilon}(x,u_n(x))u_n(x)} + \abs{g_{\varepsilon}(x,u(x))u(x)} \right)\: dx \\
&\qquad \leq 2 \limsup_{n \rightarrow \infty} \int_{\R^N\setminus A_{\lambda}} \left( \varepsilon^2 H +\mu V  \right)
\left(u_n^2
+ u^2 \right) \\
&\qquad \leq 4(1+\mu) \delta,
\end{align*}
since $\lambda \geq \Bar{\lambda}$ and $\lambda \geq \Tilde{\lambda}$. 

Since $\delta > 0$ is arbitrary, we conclude
\begin{align*}
 \lim_{n \rightarrow \infty}\norm{u_n - u}_{\varepsilon} = 0,
\end{align*}
which ends the proof.
\end{proof}

We can now state an existence theorem for the penalized problem \eqref{p2}. The proof follows from standard arguments.
\begin{theorem}\label{Th1}
 Let $g: \R \times \R^+ \rightarrow \R$ be a Carath\'eodory function satisfying $(g_1)-(g_4)$ and $V \in C(\R^N \backslash
\left\{ 0 \right\})$ be a nonnegative function.
Then, for all $\varepsilon > 0$, the functional $J_{\varepsilon}$ possesses a nontrivial critical point $u_{\varepsilon} \in
H^1_{V,\mathcal{H}}(\R^N)$, which is characterized by
\begin{align}\label{ceps}
 c_{\varepsilon} := J_{\varepsilon}(u_{\varepsilon}) = \inf_{u \in H^1_{V,\mathcal{H}}(\R^N)\setminus \lbrace0\rbrace} \max_{t >
0} J_{\varepsilon}(tu).
\end{align}
\end{theorem}

The function $u_{\varepsilon}$ found in Theorem \ref{Th1} is called a \textit{least energy solution} of \eqref{p2}. By
standard regularity theory, if $u \in H^1_{\textnormal{loc}}(\R^N)$ is a solution of \eqref{p2}, then $u \in
W^{2,q}_{\textnormal{loc}}(\R^N)$ for every $q \in (1,\infty)$. In particular, $u \in C^{1,\alpha}_{\textnormal{loc}}(\R^N)$ for
every $\alpha \in (0,1)$. Since $g_{\varepsilon}$ is not continuous, we cannot achieve a better regularity. Notice also
that, by the strong maximum principle, any nontrivial nonnegative solution $u \in C^{1,\alpha}_{\textnormal{loc}}(\R^N)$ of
\eqref{p2} is positive in $\R^N$.

\section{Asymptotics of solutions}\label{sec:Asympt}

In this section we study the asymptotic behaviour as $\varepsilon \to 0$ of the solution found in Theorem \ref{Th1}. We
follow closely the arguments in \cite{BVS}*{\S 6}. We first prove an energy estimate which is the counterpart of \cite{BVS}*{Lemma
12}. Let $\R^{N-k}_+ := \R^{N-k-1} \times \R^+$.

\begin{proposition}[Upper estimate of the critical value]\label{estim:inf}
 Suppose that the assumptions of Theorem \ref{Th1} are satisfied. For $\varepsilon$ small enough, the critical value
$c_{\varepsilon}$ defined in \eqref{ceps} satisfies
 \[\label{estimc}
 c_{\varepsilon} \leq \varepsilon^{N-k} \bigl( \omega_k \inf_{\Lambda \cap \mathcal{H}^\perp} \mathcal{M} + o(1) \bigr)\ \text{as}\ \varepsilon \rightarrow 0,
 \]
where $\omega_k$ is the volume of the unit sphere in $\R^{k+1}$.
 Moreover, the solution $u_{\varepsilon}$ of \eqref{p2} found in Theorem \ref{Th1} satisfies, for some $C>0$,
\[
 \norm{u_{\varepsilon}}_{\varepsilon}^2 \leq C\varepsilon^{N-k}.
\]
\end{proposition}

\begin{proof}
Let $x_0=(0, x''_0) \in \Lambda \cap \mathcal{H}^\perp$ be such that $\mathcal{M}(x_0)=\inf_{\Lambda \cap \mathcal{H}^\perp} \mathcal{M}$.
Denote by $I_0$ the functional defined by \eqref{defIab} with $a = V(x_0)$ and $b = K(x_0)$ and $w$ a ground state of \eqref{plim}. Take $\eta \in \mathcal{D}\left( \R^{N-k}_+ \right)$ to be a cut-off function
such that $0 \leq \eta \leq 1$, $\eta = 1$ in a neighbourhood of $(0, 	\abs{x''_0})$ and $\norm{\nabla\eta}_{\infty} \leq
C$. Consider the test function
\begin{align*}
 u(x) := \eta(x',\abs{x''}) w\left(\frac{x'}{\varepsilon}, \frac{\abs{x''}-\abs{x''_0}}{\varepsilon}\right).
\end{align*}
Setting
\begin{align*}
 u(x',x'') =: v\left(\frac{x'}{\varepsilon}, \frac{\abs{x''}-\abs{x''_0}}{\varepsilon}\right),
\end{align*}
we compute by a change of variable
\begin{multline*}
 J_{\varepsilon}(tu) \\
= \omega_k \frac{t^2}{2} \int_{\R^{N-k-1}} \int_{-\frac{\abs{x''_0}}{\varepsilon}}^{\infty} \left( \abs{\nabla v}^2 +
V(\varepsilon y',\varepsilon \rho + \abs{x''_0}) v^2 \right) (\varepsilon \rho +\abs{x''_0})^{k} \varepsilon d\rho\:
\varepsilon^{N-k-1} dy'\\ - \omega_k \int_{\R^{N-k-1}} \int_{-\frac{\abs{x''_0}}{\varepsilon}}^{\infty} G(\varepsilon
y',\varepsilon \rho + \abs{x''_0},tv) (\varepsilon \rho + \abs{x''_0})^{k} \varepsilon\: d\rho\: \varepsilon^{N-k-1} dy'.
\end{multline*}
For $\varepsilon$ small enough, we obtain 
\begin{align}\label{estimc1}
 \varepsilon^{-(N-k)} J_{\varepsilon}(tu) \leq \omega_k \abs{x''_0}^{k} I_0(tw) + o(1).
\end{align}
We deduce from \eqref{ceps} that
\[
\begin{split}
 \varepsilon^{-(N-k)} c_{\varepsilon} &\leq \max_{t > 0} \varepsilon^{-(N-k)} J_{\varepsilon}(tu)  \\
&\leq \omega_k \abs{x''_0}^{k} \max_{t > 0} I_0(tw) + o(1)  \\
&= \omega_k \mathcal{M}(x_0)+ o(1),
\end{split}
\]
which is the desired conclusion.
\end{proof}

\begin{proposition}[No uniform convergence to $0$ in $\Lambda$]
 Suppose that the assumptions of Theorem \ref{Th1} are satisfied and let $(u_{\varepsilon})_{\varepsilon} \subset
H^1_{V,\mathcal{H}}(\R^N)$ be positive solutions of \eqref{p2} obtained in Theorem \ref{Th1}. Then there exists $\delta
> 0$ such that 
\begin{align*}
 \norm{u_{\varepsilon}}_{L^{\infty}(\Lambda)} \geq \delta.
\end{align*}
\end{proposition}
\begin{proof}
 Suppose by contradiction that $\norm{u_{\varepsilon}}_{L^{\infty}(\Lambda)} \to 0$ as $\varepsilon \to 0$. Then, $(f_1)$
implies that, for all $\varepsilon$ sufficiently small, $K f(u_{\varepsilon}) \leq \mu V u_{\varepsilon}$ in $\Lambda$. By
$(g_3)$, we deduce that
\begin{align*}
 -\varepsilon^2 \left( \Delta u_{\varepsilon} + H u_{\varepsilon} \right) + (1 - \mu)V u_{\varepsilon} \leq 0 \hspace{0.5cm}
\text{in}\ \R^N.
\end{align*}
Proposition~\ref{Th:comp} then implies that $u_{\varepsilon} \equiv 0$ for all $\varepsilon$ sufficiently small, which is
impossible.
\end{proof}

 By the symmetry imposed on the solution $u_{\varepsilon}$, one can write $u_{\varepsilon}(x',x'') =
\Tilde{u}_{\varepsilon}(x',\abs{x''})$ with $\Tilde{u}_{\varepsilon} : \R^{N-k}_+ \to \R$. Since the $H^1_V$-norm of $u_{\varepsilon}$ is of the order $\varepsilon^{(N-k)/2}$, it
is natural to rescale $\Tilde{u}_{\varepsilon}(x',\abs{x''})$ as $\Tilde{u}_{\varepsilon}(x'_{\varepsilon} + \varepsilon y',
\abs{x''_{\varepsilon}} + \varepsilon \abs{y''})$ around a well-chosen family of points $x_{\varepsilon} =
(x'_{\varepsilon},x''_{\varepsilon}) \in \R^N$.

The next lemma shows that the sequences of rescaled solutions converge, up to a subsequence, in
$C^1_{\mathrm{loc}}(\R^{N-k})$ to a function $v \in H^1(\R^{N-k})$. 
\begin{lemma}
\label{convergentSubsequence}
Suppose that the assumptions of Theorem \ref{Th1} are satisfied. Let $u_{\varepsilon}\in H^1_{V,\mathcal{H}}(\R^N)$ be
positive solutions of \eqref{p2} found in Theorem \ref{Th1}, $(\varepsilon_n)_n \subset \R^+$ and $(x_n)_n \subset
\R^N$ be sequences such that $\varepsilon_n \to 0$ and $x_n = (x'_n,x''_n) \to \bar{x} =
(\bar{x}',\bar{x}'') \in \Bar{\Lambda}$ as $n\to \infty$. Set 
\[
  \Omega_n := \R^{N-k-1} \times \Bigl] -\frac{\abs{x''_n}}{\varepsilon_n}, +\infty \Bigr[
\] 
and let  $v_n : \Omega_n \to \R$ be defined by
\begin{align}\label{def:vn}
 v_n(y, z) := \Tilde{u}_{\varepsilon_n}(x'_{n} + \varepsilon_n y, \abs{x''_{n}} + \varepsilon_n z),
\end{align}
where $\Tilde{u}_{\varepsilon_n} : \R^{N-k}_+ \to \R$ is such that $u_{\varepsilon_n}(x',x'')
=\Tilde{u}_{\varepsilon_n}(x', \abs{x''})$. Then, there exists $v \in H^1(\R^{N-k})$ such that, along a subsequence that we still denote by $(v_n)_n$,
\begin{align*}
 v_n \stackrel{C^1_{\mathrm{loc}}(\R^{N-k})}{\longrightarrow} v.
\end{align*}
\end{lemma}
\begin{proof}
 First observe that each $v_n$ solves the equation
\begin{align}\label{conv:eqn}
 -\Delta v_n - \frac{\varepsilon_n k}{z} \Dp{v_n}{z} + V(y_{n} + \varepsilon_n y, \abs{x'_{n}} + \varepsilon_n z) v_n = 
g_{\varepsilon_n} (y_{n} + \varepsilon_n y, \abs{x'_{n}} + \varepsilon_n z, v_n), 
\end{align}
in $\Omega_n$. We infer from Proposition \ref{estim:inf} that for all $n \in \N$,
\begin{align*}
 \int_{\Omega_n} \left( \abs{\nabla v_n(y,z)}^2 + V(y_{n} + \varepsilon_n y, \abs{x'_{n}} + \varepsilon_n z)
\abs{v_n(y,z)}^2 \right)\: dy\, dz \leq C,
\end{align*}
with $C>0$ independent of $n$.

Define a cut-off function $\eta_R \in \mathcal{D}(\R^{N-k})$ such that $0\leq \eta_R \leq 1$, $\eta_R(x) = 1$ if
$\abs{x} \leq R/2$, $\eta_R(x)=0$ if $\abs{x}\geq R$ and $\norm{\nabla\eta_R}_{\infty} \leq C/R$ for some $C>0$. Choose
$(R_n)_n$ such that $R_n\to\infty$ and $\varepsilon_n R_n \to 0$. Since $\Bar{x} \in \Lambda$ and $\Bar{\Lambda} \cap 
\mathcal{H}=\emptyset$, one has $\varepsilon_n R_n \leq \abs{x''_{n}}$ if $n$ is large enough. 
Define $w_n \in H^1_{\mathrm{loc}}(\R^{N-k})$ by
\[
 w_n(y) := \eta_{R_n}(y) v_n(y).
\]

On the one hand, we notice that
\[
 \begin{split}
 \int_{\R^{N-k}} w_n^2 &\leq \int_{B(0,R_n)} v_n^2  \\
&\leq \frac{1}{\inf_{B(x_n,\varepsilon_n R_n)} V} \int_{\Omega_n} V(y_{n}
+ \varepsilon_n y, \abs{x'_{n}} + \varepsilon_n z) \abs{v_n(y,z)}^2\: dy dz.
 \end{split}
\]
Since $V$ is positive on $\Bar{\Lambda}$ and continuous on $\R^N$, the convergence of $x_n$ to a point in $\Bar{\Lambda}$ implies that 
\begin{align}\label{conv:estim1}
 \int_{\R^{N-k}} w_n^2 \leq C.
\end{align}
On the other hand, we compute in the same way as in \cite{BVS}*{Lemma 13}
\begin{align}\label{conv:estim2}
 \int_{\R^{N-k}} \abs{\nabla w_n}^2 \leq C \norm{v_n}^2_{H^1(B(0,R_n))}.
\end{align}
Since
\begin{align*}
 \norm{v_n}_{H^1(B(0,R_n))} \leq C \norm{u_{\varepsilon_n}}_{\varepsilon},
\end{align*}
we deduce from \eqref{conv:estim1} and \eqref{conv:estim2} that $(w_n)_n$ is bounded in $H^1(\R^{N-k})$. Since $w_n$
solves equation \eqref{conv:eqn} on $B(0,R_n)$ for all $n$, classical regularity estimates yield that for every $R>0$
and every $q>1$,
\begin{align}\label{conv:estimW2q}
 \sup_{n\in \N} \norm{v_n}_{W^{2,q}(B(0,R))} < \infty.
\end{align}

Up to a subsequence, we can now assume that $(w_n)_n$ converges weakly in $H^1(\R^{N-k})$ to some function $v \in
H^1(\R^{N-k})$. By \eqref{conv:estimW2q}, for every
compact $K \subset \R^{N-k}$, $w_n$ converges to $v$ in $C^1(K)$. Moreover, for $n$ large enough, $w_n = v_n$ in $K$ so that $v_n\to v$ in
$C^1(K)$.
\end{proof}

In the next two lemmas, we will estimate from below the action of $u_{\varepsilon}$ inside and outside neighbourhoods of
points. Since we expect the concentration set to be a $k$-sphere in $\R^N$, the following distance will be useful. For
$x,y \in \R^N$, let
\begin{align*}
 d_{\mathcal{H}}(x,y) := \sqrt{\abs{x'-y'}^2+ \left( \abs{x''}-\abs{y''}\right) ^2}.
\end{align*}
Thus $d_{\mathcal{H}}(x,y)$ represents the distance between the $k$-spheres centered at $x'$ and $y'$, and of radius
$\abs{x''}$ and $\abs{y''}$ respectively. 
We denote by $B_{\mathcal{H}}$ the balls for the distance $d_{\mathcal{H}}$, i.e.,
\[
 B_{\mathcal{H}}(x,r)=\{ y \in \R^N : d_{\mathcal{H}}(x, y) < r\}.
\]

\begin{lemma}
Suppose that the assumptions of Theorem \ref{Th1} are satisfied. Let $u_{\varepsilon}\in H^1_{V,\mathcal{H}}(\R^N)$ be
positive solutions of \eqref{p2} found in Theorem \ref{Th1}, $(\varepsilon_n)_n \subset \R^+$ and $(x_n)_n \subset
\R^n$ be sequences such that $\varepsilon_n \to 0$ and $x_n = (x'_n,x''_n) \to \bar{x} =
(\bar{x}',\bar{x}'') \in \Bar{\Lambda}$ as $n\to \infty$. If 
\begin{align}\label{estim:sup1}
 \liminf_{n\to\infty} u_{\varepsilon_n}(x_n) > 0,
\end{align}
then we have, up to a subsequence,
 \begin{align*}
 \liminf_{R\to\infty}\liminf_{n\to\infty} \varepsilon_n^{-(N-k)} \left( \int_{T_n(R)}  \frac{1}{2} \left( \varepsilon_n^2
\abs{\nabla u_{\varepsilon_n}}^2 + V u_{\varepsilon_n}^2 \right) - G_{\varepsilon_n}(x,u_{\varepsilon_n}) \right) 
\geq \omega_k \mathcal{M}(\bar{x}),
 \end{align*}
where $T_n(R) := B_{\mathcal{H}}(x_n,\varepsilon_n R)$.
\end{lemma}
\begin{proof}
Let $v_n$ be defined by \eqref{def:vn}. Passing to a subsequence if necessary, we may assume that there exists $v \in
H^1(\R^{N-k})$ such that $v_n \to v$ in $C^1_{\mathrm{loc}}(\R^{N-k})$. Since $\Lambda$ is smooth, we can also assume that
the sequence of characteristic functions $\chi_n(y, z) = \chi_{\Lambda}(x'_n+\varepsilon_n y,\abs{x_n''}+\varepsilon_n z)$
converges almost everywhere to a measurable function $\chi$ satisfying $0\leq \chi \leq 1$. We then deduce that $v$
solves the limiting equation
\[
 -\Delta v + V(\bar{x})v = \Tilde{g}(y, v) \hspace{1cm} \text{in} \ \R^{N-k},
\]
where
\begin{align*}
 \Tilde{g}(y, s) := \chi(y) K(\bar{x}) f(s) + \left( 1-\chi(y) \right) \min\left\lbrace \mu V(\bar{x}) s, K(\bar{x})f(s)
\right\rbrace.
\end{align*}
 By \eqref{estim:sup1}, we know that $v(0) = \lim_{n\to\infty} v_n(0) > 0$, so that $v$ is not identically zero.

It was shown in \cite{BVS}*{Lemma 14} that
\begin{multline*}
 \liminf_{R\to\infty}\liminf_{n\to\infty} \int_{B(0,R)} \biggl( \frac{1}{2} \left( \abs{\nabla v_n(y, z)}^2 + V(x'_n +
\varepsilon_n y, \abs{x''_n} + \varepsilon_n z)\abs{v_n(y, z)}^2 \right) \\ \shoveright{- G_{\varepsilon_n}\left(x'_n + \varepsilon_n
y, \abs{x''_n} + \varepsilon_n z,v_n(y, z)\right) \biggr)\: dz\: dy} \\ \qquad \geq \frac{1}{2} \int_{\R^{N-k}} \left( 
\abs{\nabla v}^2 + V(\bar{x}) v^2 \right) - \int_{\R^{N-k}} \Tilde{G}(y,v(y))\: dy,
 \end{multline*}
where $\Tilde{G}(x,s) := \int_0^s \Tilde{g}(x,\sigma)\: d\sigma$.

Set $B_n(R) := B((x',\abs{x''}),\varepsilon_n R) \subset \R^{N-k}$. By a computation similar to the one
leading to \eqref{estimc1}, we have
\begin{multline*}
\int_{T_n(R)} \left( \frac{1}{2} \left( \varepsilon_n^2 \abs{\nabla u_{\varepsilon_n}(x)}^2 +
V(x)\abs{u_{\varepsilon_n}(x)}^2 \right) - G_{\varepsilon_n}(x,u_{\varepsilon_n}(x)) \right)\: dx \\ 
 = \omega_k \int_{B_n(R)} \biggl( \frac{1}{2} \left( \varepsilon_n^2 \abs{\nabla \Tilde{u}_{\varepsilon_n}(x',r)}^2 +
V(x',r)\abs{\Tilde{u}_{\varepsilon_n}(x',r)}^2 \right)\\
\shoveright{- G_{\varepsilon_n}(x', r,\Tilde{u}_{\varepsilon_n}(x',r)) \biggr) r^k\:
dr\: dx'} \\
 = \omega_k \abs{\bar{x}''}^k \varepsilon_n^{N-k} \int_{B(0,R)} \biggl( \frac{1}{2} \left( \abs{\nabla v_n(y, z)}^2 +
V(x'_n + \varepsilon_n y, \abs{x''_n} + \varepsilon_n z)\abs{v_n(y, z)}^2 \right) \\ 
- G_{\varepsilon_n}(x'_{n} +
\varepsilon_n y, \abs{x''_{n}} + \varepsilon_n z,v_n(y, z)) \biggr)\: dz\: dy + o(1).
 \end{multline*}
The conclusion follows.
\end{proof}

\begin{lemma}
 Suppose that the assumptions of Theorem \ref{Th1} are satisfied. Let $u_{\varepsilon}\in H^1_{V,\mathcal{H}}(\R^N)$ be
positive solutions of \eqref{p2} found in Theorem \ref{Th1}, $(\varepsilon_n)_n \subset \R^+$ and $(x^i_n)_n \subset
\R^N$ be sequences such that $\varepsilon_n \to 0$ and for $1\leq i \leq M$, $x^i_n\to
\bar{x}^i \in \Bar{\Lambda}$ as $n\to \infty$. Then, up to a subsequence, we have
\begin{align*}
  \liminf_{R\to\infty}\liminf_{n\to\infty} \varepsilon_n^{-(N-k)} \left( \int_{\R^N\setminus \mathcal{T}_n(R)} 
\frac{1}{2} \left( \varepsilon_n^2 \abs{\nabla u_{\varepsilon_n}}^2 + V u_{\varepsilon_n}^2 \right) -
G_{\varepsilon_n}(x,u_{\varepsilon_n}) \right) \geq 0,
\end{align*}
where $\mathcal{T}_n(R) := \bigcup_{i=1}^K B_{\mathcal{H}}(x^i_n,\varepsilon_n R)$.
\end{lemma}
\begin{proof}
 See \cite{BVS}*{Lemma 15}.
\end{proof}

\begin{proposition}[Lower estimate of the critical value]\label{estim:sup}
 Suppose that the assumptions of Theorem \ref{Th1} are satisfied. Let $u_{\varepsilon}\in H^1_{V,\mathcal{H}}(\R^N)$ be
positive solutions of \eqref{p2} found in Theorem \ref{Th1}, $(\varepsilon_n)_n \subset \R^+$ and $(x^i_n)_n \subset \R^N$
be sequences such that $\varepsilon_n \to 0$ and for $1\leq i \leq M$, $x^i_n\to
\bar{x}^i \in \Bar{\Lambda}$ as $n\to \infty$. If for every $1\leq i < j \leq M$, we have
\begin{align*}
 \limsup_{n\to\infty} \frac{d_{\mathcal{H}}(x^i_n,x^j_n)}{\varepsilon_n} = \infty
\end{align*}
and if for every $1\leq i \leq M$,
\[
 \liminf_{n\to\infty} u_{\varepsilon_n}(x^i_n) > 0,
\]
then the critical value $c_{\varepsilon}$ defined in \eqref{ceps} satisfies
 \begin{align*}
 \liminf_{n\to\infty} \varepsilon_n^{-(N-k)} c_{\varepsilon_n} \geq \omega_k \sum_{i=1}^M \mathcal{M}(\bar{x}^i).
 \end{align*}
\end{proposition}
\begin{proof}
This is a consequence of the two previous lemmas, see \cite{BVS}*{Proposition 16} for the details.
 \end{proof}

The following proposition is the key result for the next section.
\begin{proposition}[Uniform convergence to $0$ outside small balls]\label{Prop1}
 Suppose that the assumptions of Theorem \ref{Th1} are satisfied and that $\Lambda$ satisfies the assumptions of
Section \ref{sect:Lambda}.
Let $(u_{\varepsilon})_{\varepsilon} \subset H^1_{V,\mathcal{H}}(\R^N)$ be positive solutions of \eqref{p2} obtained in
Theorem \ref{Th1}. If $(x_{\varepsilon})_{\varepsilon>0} \subset \Lambda$ is such that
\begin{align*}
 \liminf_{\varepsilon\to 0} u_{\varepsilon}(x_{\varepsilon}) > 0,
\end{align*}
then
\begin{enumerate}[(i)]
 \item $\lim_{\varepsilon\to 0} \mathcal{M}(x_{\varepsilon}) = \inf_{\Lambda \cap \mathcal{H}^\perp} \mathcal{M}$,
 \item $\lim_{\varepsilon \to 0} \frac{\dist(x_\varepsilon, \mathcal{H}^\perp)}{\varepsilon}=0$,
 \item $\liminf_{\varepsilon \to 0} d_{\mathcal{H}}(x_{\varepsilon}, \partial \Lambda) > 0$,
 \item for every $\delta>0$, there exists $\varepsilon_0>0$ and $R>0$ such that, for every $\varepsilon \in (0,\varepsilon_0)$,
\[
 \norm{u_{\varepsilon}}_{L^{\infty}\left(\Lambda\setminus B_{\mathcal{H}}(x_{\varepsilon},\varepsilon R)\right)} \leq \delta.
\]
\end{enumerate}
\end{proposition}
\begin{proof}
The first assertion is a direct consequence of Propositions \ref{estim:inf} and \ref{estim:sup}, see \cite{BVS}*{Proposition
33} for the details.

\medbreak

For the second assertion, since $\Bar{\Lambda}$ is compact, we can assume by contradiction that there exist sequences
$(\varepsilon_n)_n \subset \R^+$ and
$(x_n)_n \subset \R^{N}$ such that $\varepsilon_n \to 0$,
\[
 \liminf_{n\to\infty} u_{\varepsilon_n}(x_n) > 0,
\]
and 
\[
  x_n \to \Bar{x} \in \Bar{\Lambda}\setminus \mathcal{H}^\perp.
\]

If $k=N-2$, let $R \in \mathbf{O}(N)$ denote the reflexion with respect to $\mathcal{H}^\perp$. By definition of $H^1_{V, \mathcal{H}}(\R^N)$, $u \circ R=u$, and thus
\[
 \liminf_{n\to\infty} u_{\varepsilon_n}(R(x_n)) > 0.
\]

Since $d_{\mathcal{H}}(\Bar{x}, R(\Bar{x}))> 0$, one has $\lim_{n \to \infty} \frac{d_{\mathcal{H}}(x_n,
R(x_n))}{\varepsilon_n} = \infty$.
By Proposition \ref{estim:sup}, we obtain 
\[
 \liminf_{n\to\infty} \varepsilon_n^{-(N-k)} c_{\varepsilon_n} \geq \omega_k \left( \mathcal{M}(\bar{x}) +
\mathcal{M}(R(\bar{x})) \right) \geq 2\omega_k \inf_{\Lambda} \mathcal{M}.
\]
which, together with Proposition \ref{estim:inf}
\[
 \liminf_{n\to\infty} \varepsilon_n^{-(N-k)} c_{\varepsilon_n} \leq \omega_k \inf_{\Lambda \cap \mathcal{H}^\perp} \mathcal{M}
\]
is in contradiction with \eqref{def:Lambda}.

In the case where $k < N-2$, since $\inf_{\Lambda} \mathcal{M} > 0$, choose $\ell \in \N$ such that 
\begin{equation}
\label{Prop1:m0}
  \inf_{\Lambda \cap \mathcal{H}^\perp} \mathcal{M} < \ell\inf_{\Lambda} \mathcal{M}.
\end{equation}
There exist isometries $R_1, \dots, R_l$ of $\R^N$ such that $R_{i}(\mathcal{H})=\mathcal{H}$ and $R_i(\bar{x})
\neq R_j(\bar{x})$, for every $i,j \in \{1, \dots, \ell\}$ with $i \ne j$. One has hence 
\[
  \lim_{n \to \infty} \frac{d_{\mathcal{H}}(R_i(x_n), R_j(x_n))}{\varepsilon}=0
\]
By Proposition \ref{estim:sup}, we get
\begin{align*}
 \liminf_{n\to\infty} \varepsilon_n^{-(N-k)} c_{\varepsilon_n} \geq \omega_k \sum_{i=1}^{\ell} \mathcal{M}(R_i(\bar{x}))  \geq
l\omega_k \inf_{\Lambda} \mathcal{M},
\end{align*}
so that, in view of the upper estimate of Proposition~\ref{estim:inf}, we have a contradiction with \eqref{Prop1:m0}.

\medbreak

For the third assertion, suppose by contradiction that there exist sequences $(\varepsilon_n)_n \subset \R^+$ and
$(x_n)_n \subset \R^{N}$ such that $\varepsilon_n \to 0$,
\begin{align*}
 \liminf_{n\to\infty} u_{\varepsilon_n}(x_n) > 0,
\end{align*}
and, $x_n \to \Bar{x} \in \partial \Lambda$. We have just proven that $\Bar{x} \in \mathcal{H}^\perp$.
By Proposition \ref{estim:sup}, we have
 \begin{align*}
 \liminf_{n\to\infty} \varepsilon_n^{-(N-k)} c_{\varepsilon_n} \geq \omega_k \mathcal{M}(\bar{x}) \geq \omega_k
\inf_{\partial \Lambda \cap \mathcal{H}^\perp} \mathcal{M}.
 \end{align*}
This inequality, along with Proposition~\ref{estim:inf}, contradicts \eqref{hyp:Lambda}.

\medbreak
In order to obtain the last assertion, suppose by contradiction that there exist sequences $(\varepsilon_n)_n \subset
\R^+$, $(x_n)_{n}$  and $(y_n)_n
\subset \Lambda$ such that $\varepsilon_n \to 0$,
\begin{align*}
 u_{\varepsilon_n}(y_n) \geq \delta,
\end{align*}
and
\begin{align*}
 \lim_{n \to \infty} \frac{d_{\mathcal{H}}(x_{n},y_n)}{\varepsilon_n} = \infty. 
\end{align*}
Up to a subsequence, we can assume that $x_{n}\to \bar{x} \in \Lambda$ and $y_n\to
\bar{y} \in \Lambda$. In view of the second assertion, one has $\Bar{x} \in \mathcal{H}^\perp$ and $\Bar{y} \in
\mathcal{H}^\perp$. Therefore, by Proposition \ref{estim:sup},
\begin{align*}
 \liminf_{n\to\infty} \varepsilon_n^{-(N-k)} c_{\varepsilon_n} \geq \omega_k \left( \mathcal{M}(\bar{x}) +
\mathcal{M}(\bar{y}) \right) \geq 2\omega_k \inf_{\Lambda \cap \mathcal{H}^\perp} \mathcal{M}.
 \end{align*}
In view of the assumption of \eqref{hyp:Lambda}, this would contradict Proposition~\ref{estim:inf}.
\end{proof}
\section{Barrier functions}\label{sec:barrier}

\subsection{Linear inequation outside small balls}
In this section we prove that for $\varepsilon$ small enough, the solutions of the penalized problem \eqref{p2} are also
solutions of the initial problem \eqref{p1}. We follow the arguments of \cite{MVS}. First we notice that the solutions
of \eqref{p2} satisfy a linear inequation outside small balls.
\begin{lemma}\label{lemma:subsol}
 Suppose that the assumptions of Proposition \ref{Prop1} are satisfied and let $(u_{\varepsilon})_{\varepsilon>0} \subset
H^1_{V,\mathcal{H}}(\R^N)$ be positive solutions of \eqref{p2} found in Theorem \ref{Th1} and
$(x_{\varepsilon})_{\varepsilon>0} \subset \Lambda$ be such that
\begin{align*}
 \liminf_{\varepsilon\to 0} u_{\varepsilon}(x_{\varepsilon}) > 0.
\end{align*}
Then there exist $\rho > 0$ and $\varepsilon_0>0$ such that for all $\varepsilon \in (0,\varepsilon_0)$,
\begin{align}\label{lowsol}
 -\varepsilon^2 \left( \Delta u_{\varepsilon} + H u_{\varepsilon} \right) + (1 - \mu)V u_{\varepsilon} \leq 0 \hspace{0.5cm}
\text{in}\ \R^N\setminus B_\mathcal{H}(x_{\varepsilon},\varepsilon R).
\end{align}
\end{lemma}
\begin{proof}
 Set
\begin{align*}
 \eta := \inf_{x\in \Lambda} \frac{\mu V(x)}{K(x)}.
\end{align*}
Since $V$ and $K$ are bounded positive continuous functions on $\Bar{\Lambda}$, $\eta > 0$. By $(f_1)$, there exists
$\delta > 0$ such that 
\begin{align*}
 \frac{f(s)}{s} \leq \eta \hspace{0.5cm} \text{for all} \ s \leq \delta.
\end{align*}
By Proposition \ref{Prop1}, we can find $\varepsilon_0 > 0$ and $\rho >0$ such that for all $\varepsilon \in (0,\varepsilon_0]$,
one
has
\begin{align*}
 u_{\varepsilon}(x) \leq \delta \hspace{0.5cm} \text{for all} \ x \in \Lambda \setminus
B_\mathcal{H}(x_{\varepsilon},\varepsilon \rho).
\end{align*}
Hence
\begin{align*}
 K(x) f(u_{\varepsilon}(x)) \leq \mu V(x) u_{\varepsilon}(x) \hspace{0.5cm} \text{in} \ \Lambda \setminus
B_\mathcal{H}(x_{\varepsilon},\varepsilon \rho).
\end{align*}
We conclude that
\begin{align*}
 -\varepsilon^2 \Delta u_{\varepsilon} + (1 - \mu)V u_{\varepsilon} \leq -\varepsilon^2 \Delta u_{\varepsilon} + V u_{\varepsilon} - K
f(u_{\varepsilon}) = 0 \hspace{0.5cm} \text{in} \ \Lambda \setminus B_\mathcal{H}(x_{\varepsilon},\varepsilon \rho).
\end{align*}
The fact that $u_{\varepsilon}$ satisfies \eqref{lowsol} in $\R^N \setminus\Lambda$ follows directly from the definition of
the
penalized nonlinearity.
\end{proof}

This lemma suggests that we can compare the solution $u_{\varepsilon}$ with supersolutions of the operator $-\varepsilon^2
\left( \Delta + H \right) + (1 - \mu)V$ in order to obtain decay estimates of $u_{\varepsilon}$. 

\subsection{Comparison functions}

The next lemma provides a minimal positive solutions of the operator $- \Delta - H$ in $\R^N\setminus
\Bar{\Lambda}$.

\begin{lemma}
\label{lemmaPsiepsilon}
For every $\varepsilon > 0$, there exists $\Psi_\varepsilon \in C^2\left((\R^N\setminus \lbrace 0 \rbrace) \setminus \Lambda
\right)$ such that
\[
 \left\{ \begin{aligned}
          -\varepsilon^2(\Delta \Psi_\varepsilon + H \Psi_\varepsilon)+(1-\mu)V\Psi_\varepsilon &= 0 & &\text{in}\ \R^N\setminus
\Bar{\Lambda}, \\
	\Psi_{\varepsilon} &= 1 & &\text{on}\ \partial \Lambda,
         \end{aligned}
 \right.
\]
and
\begin{equation}
\label{PsiepsilonFiniteIntegral}
 \int_{\R^N\setminus \Lambda} \biggl( \abs{\nabla \Psi_\varepsilon(x)}^2 + \frac{\abs{\Psi_\varepsilon(x)}^2}{\abs{x}^2}
\biggr)\: dx <
\infty.
\end{equation}
Moreover, there exists $C > 0$ such that, for every $x \in \R^N\setminus \Lambda$ and every $\varepsilon > 0$, 
\begin{equation}
\label{DecayPsiepsilon}
 0 < \Psi_\varepsilon(x) \leq \frac{C}{(1+\abs{x})^{N-2}}.
\end{equation}
\end{lemma}
\begin{proof}
The function $\Psi_\varepsilon$ is obtained by minimimizing 
\[
 \int_{\R^N\setminus \Lambda} \left( \varepsilon^2 \left( \abs{\nabla u}^2 - H u^2 \right) + (1-\mu) V u^2 \right)\: dx
\]
on the set
\[
 \{ u \in H^1_V(\R^N)\: :\: u=1 \text{ on } \partial \Lambda\}.
\]
By classical elliptic regularity theory, $\Psi_\varepsilon \in C^2\left((\R^N\setminus \lbrace 0 \rbrace) \setminus \Lambda
\right)$.
The estimate \eqref{PsiepsilonFiniteIntegral} follows from \eqref{positivity}.

In order to obtain the estimate \eqref{DecayPsiepsilon} consider the problem
\[
 \left\{ \begin{aligned}
          -\Delta \Psi - H \Psi &= 0 & &\text{in}\ \R^N\setminus \Bar{\Lambda}, \\
	\Psi &= 1 & &\textnormal{on}\ \partial \Lambda.
         \end{aligned}
 \right.
\]
We have just proved that this problem has a solution $\Psi \in C^2((\R^N \setminus \{0\})\setminus \Lambda)$ such that
\begin{equation}
\label{PsiFiniteIntegral}
 \int_{\R^N\setminus \Lambda} \biggl( \abs{\nabla \Psi(x)}^2 + \frac{\abs{\Psi(x)}^2}{\abs{x}^2} \biggr)\: dx <
\infty.
\end{equation}

Now set for $\rho \in (0,1)$ and $x \in B(0, \rho)$,
\begin{align*}
 W(x) := (N-2)\beta - \kappa \left( \log\frac{1}{\abs{x}} \right)^{-\beta},
\end{align*}
We compute that
\begin{align*}
 -\Delta W(x) = \frac{\kappa \beta}{\abs{x}^2} \left[ (N-2) \left( \log\frac{1}{\abs{x}}\right)^{-(1+\beta)} + (1+\beta)
\left( \log\frac{1}{\abs{x}}\right)^{-(2+\beta)} \right].
\end{align*}
Since for $\abs{x} \leq 1$,
\[
H(x)\leq \frac{\kappa}{(\abs{x}^2 \log \frac{1}{\abs{x}})^{1+\beta}}
\]
the function $W$ is a supersolution of $-\Delta - H$ in $B(0,1)$. 
Moreover, if one takes $\rho < 1$ such that 
\[
  (N-2)\beta \left( \log\frac{1}{\rho} \right)^{\beta} > \kappa,
\]
$W$ is positive on $\partial B(0,\rho)$. 
In view of \eqref{PsiFiniteIntegral} Proposition~\ref{Th:comp} implies that $\Psi$ is bounded from above by
a positive multiple of $W$ in $B(0,\rho)$. Since $\Psi$ is continuous and $W$ is bounded in $B(0,1)$, we obtain that
$\Psi$ is bounded in $B(0, 1)$.
By similarly considering
\[
 W(x) := \frac{1}{\abs{x}^{N-2}} \left((N-2)\beta - \kappa \left( \log\abs{x} \right)^{-\beta}\right)
\]
(see Lemma 3.4 of \cite{MVS}), we obtain that $\Psi(x)\sim\abs{x}^{N-2}$. We have thus proven that
\[
 \Psi(x)\leq \frac{C}{(1+\abs{x})^{N-2}}.
\]

Now, note that since $V$ is nonnegative,
\[
 -\Delta \Psi_\varepsilon -H \Psi_\varepsilon \leq 0.
\]
In view of \eqref{PsiFiniteIntegral} and \eqref{PsiepsilonFiniteIntegral}, Proposition~\ref{Th:comp} is applicable, and
for every $x \in \R^N\setminus \Lambda$,
\[
 \Psi_\varepsilon(x) \leq \Psi(x) \leq \frac{C}{(1+\abs{x})^{N-2}}.\qedhere
\]
\end{proof}

As explained in \cite{MVS}, the estimate \eqref{DecayPsiepsilon} is the best one can hope for if $V$ decays rapidly at infinity
or is compactly supported. However, if $V$ decays quadratically or subquadratically at infinity, we can improve
\eqref{DecayPsiepsilon}.

\begin{lemma}
\label{lemmaImprovedDecay}
Let $\Psi_\varepsilon$ be given by Lemma~\ref{lemmaPsiepsilon}.
\begin{enumerate}[(1)]
\item \label{Psiepsiloninfinityquadratic} If $\liminf_{\abs{x}\rightarrow\infty} V(x)\abs{x}^{2} > 0$, then there exist
$\lambda > 0$, $R > 0$ and $C > 0$ such that for every $\varepsilon > 0$ and $x \in \R^N\setminus B(0,R)$, 
\[
  \Psi_\varepsilon (x) \leq
C\left(\frac{R}{\abs{x}}\right)^{\frac{N-2}{2}+\sqrt{\left(\frac{N-2}{2}\right)^2 - \kappa
+ \frac{\lambda^2}{\varepsilon^2}}}.
\]
\item \label{PsiepsiloninfinitySuperquadratic} If $\liminf_{\abs{x}\rightarrow\infty} V(x)\abs{x}^{\alpha} > 0$ with
$\alpha < 2$, then there exist $\lambda > 0$, $R > 0$, $C > 0$ and $\varepsilon_0 > 0$ such that for every 
$\varepsilon \in (0,\varepsilon_0)$ and $x \in \R^N\setminus B(0,R)$, 
\[
  \Psi_\varepsilon (x) \leq C\exp \left(-\frac{\lambda}{\varepsilon}
\left(\abs{x}^\frac{2-\alpha}{2}-R^\frac{2-\alpha}{2}\right)\right).
\]
\item  If $\liminf_{\abs{x}\rightarrow 0} V(x)\abs{x}^{2} > 0$, then there exist $\lambda > 0$, $r > 0$ and $C > 0$
such that for every $\varepsilon > 0$ and $x \in B(0,r)$, 
\[
  \Psi_\varepsilon (x) \leq
C\left(\frac{\abs{x}}{r}\right)^{\sqrt{\left(\frac{N-2}{2}\right)^2 -\kappa
+\frac{\lambda^2}{\varepsilon^2}}-\frac{N-2}{2}}.
\]
\item If $\liminf_{\abs{x}\rightarrow 0} V(x)\abs{x}^{\alpha} > 0$ with $\alpha < 2$, then there exist $\lambda >
0$, $R > 0$, $C > 0$ and $\varepsilon_0 > 0$ such that for every $\varepsilon \in (0,\varepsilon_0)$ and $x \in B(0,r)$, 
\[
  \Psi_\varepsilon (x) \leq C\exp \left(-\frac{\lambda}{\varepsilon}
\left(\abs{x}^{-\frac{\alpha-2}{2}}-r^{-\frac{\alpha-2}{2}}\right)\right).
\]
\end{enumerate}
\end{lemma}

\begin{proof}
For \eqref{Psiepsiloninfinityquadratic}, there exist $R > 0$ and $\lambda > 0$ such that for 
$x \in \R^N\setminus B(0,R)$
\[
 (1-\mu)V(x) \geq \frac{\lambda^2}{\abs{x}^2}.
\]
One then checks that
\[
 W(x)=\left(\frac{R}{\abs{x}}\right)^{\frac{N-2}{2}+\sqrt{\left(\frac{N-2}{2}\right)^2 - \kappa
+ \frac{\lambda^2}{\varepsilon^2}}}
\]
is a supersolution in $\R^N\setminus B(0,R)$.

For \eqref{PsiepsiloninfinitySuperquadratic}, there exist $R > 0$ and $\eta > 0$ such that for 
$x \in \R^N\setminus B(0,R)$
\[
 (1-\mu)V(x) \geq \frac{\eta}{\abs{x}^\alpha}.
\]
One then checks that 
\[
 W(x)=\exp \left(-\frac{\lambda}{\varepsilon} \left(\abs{x}^{\frac{2-\alpha}{2}}-R^{\frac{2-\alpha}{2}}\right)\right)
\]
is a supersolution in $\R^N\setminus B(0,R)$ with $\lambda^2 < (\frac{2}{2-\alpha})^2 \nu$ and $\varepsilon$ 
small enough.

The proofs of the other assertions are similar.
\end{proof}

The other tool is a function that describes the exponential decay of $u_\varepsilon$ inside $\Lambda$.

\begin{lemma}\label{CompPeak}
Let $\Bar{x} \in \Lambda$ and $R > 0$ be such that 
\begin{equation}
\label{condRxeps}
 B_{\mathcal{H}}(\Bar{x},R) \subset \Lambda.
\end{equation} 
Define
\begin{align}\label{def:Phieps}
\Phi_{\varepsilon}^{\Bar{x}}(x) := \cosh \left( \lambda \frac{R-d_{\mathcal{H}}(x,\Bar{x})}{\varepsilon}\right).
\end{align}
There exists $\lambda > 0$ and $\varepsilon_0 > 0$ such that for every $\varepsilon \in (0,\varepsilon_0)$, one has
\begin{align*}
 -\varepsilon^2 \Delta \Phi^{\Bar{x}}_{\varepsilon} + (1-\mu) V \Phi^{\Bar{x}}_{\varepsilon} \geq 0 \ \text{in}\
B_\mathcal{H}(\Bar{x},R).
\end{align*}
\end{lemma}
\begin{proof}
First one computes
\begin{multline*}
 -\varepsilon^2\Delta \Phi^{\Bar{x}}_\varepsilon(x)=-\lambda^2 \cosh\left(\frac{\lambda}{\varepsilon}\left(R-d_{\mathcal{H}}(x,
\Bar{x})\right)\right) \\
+\frac{\varepsilon \lambda}{d_{\mathcal{H}}(x, \Bar{x})} \left( N-1
 - k \frac{\abs{\Bar{x}''}}{\abs{x''}} \right)
\sinh\left(\frac{\lambda}{\varepsilon}\left(R-d_{\mathcal{H}}(x, \Bar{x})\right)\right).
\end{multline*}
Let us choose $\lambda > 0$ such that $\lambda^2 < (1-\mu) \inf_{\Lambda} V$. 
In view of \eqref{condRxeps}, one has for $x \in B_{\mathcal{H}}(\Bar{x}, R)$, 
\begin{multline*}
 -\varepsilon^2\Delta \Phi^{\Bar{x}}_\varepsilon(x)+ (1-\mu) V \Phi^{\Bar{x}}_\varepsilon(x) \\ 
\geq \frac{\varepsilon \lambda}{d_{\mathcal{H}}(x, \Bar{x})} \left( N-1
 - k \frac{\abs{\Bar{x}''}}{\abs{x''}} \right)
 \sinh\left(\frac{\lambda}{\varepsilon}\left(R-d_{\mathcal{H}}(x, \Bar{x})\right)\right) \\
 + \left( (1-\mu) \inf_{\Lambda} V - \lambda^2 \right) 
 \cosh\left(\frac{\lambda}{\varepsilon}\left(R-d_{\mathcal{H}}(x, \Bar{x})\right)\right).
\end{multline*}
This last expression is positive if $\varepsilon$ is sufficiently small.
\end{proof}

\begin{lemma}\label{lemma:barrier}
Let $(x_{\varepsilon})_{\varepsilon} \subset \Lambda$ be such that 
 \begin{align*}
 \liminf_{\varepsilon \to 0} d_{\mathcal{H}}(x_{\varepsilon}, \partial \Lambda) > 0
 \end{align*}
and $\rho > 0$.
Then, there exist $\varepsilon_0 > 0$ and a family of functions $(W_{\varepsilon})_{0<\varepsilon<\varepsilon_0} \subset
C^{1,1}_{\textnormal{loc}}((\R^N \setminus \{0\}) \setminus B_\mathcal{H}(x_{\varepsilon},\varepsilon \rho))$ such that for
all $\varepsilon \in (0,\varepsilon_0)$, one has
 \begin{enumerate}[(i)]
 \item $W_{\varepsilon}$ satisfies the inequation 
\begin{align*}
 -\varepsilon^2 \left( \Delta + H\right)  W_{\varepsilon} + (1-\mu) V W_{\varepsilon} \geq 0 \ \text{in}\ \R^N \setminus
B_\mathcal{H}(x_{\varepsilon},\varepsilon \rho),
\end{align*}
 \item\label{barrierb} $\nabla W_{\varepsilon} \in L^2(\R^N\setminus B_\mathcal{H}(x_{\varepsilon},\varepsilon \rho))$ and
$\frac{W_{\varepsilon}}{\abs{x}} \in L^2(\R^N\setminus B_\mathcal{H}(x_{\varepsilon},\varepsilon \rho))$,
 \item \label{barrierc} $W_{\varepsilon}\geq 1$ on $\partial B_\mathcal{H}(x_{\varepsilon},\varepsilon \rho)$,
\item \label{barrierd} for every $x \in B_\mathcal{H}(x_{\varepsilon},\varepsilon \rho)$,
\[
 W_{\varepsilon}(x) \leq C \exp{\left( -\frac{\lambda}{\varepsilon} \frac{d_{\mathcal{H}}(x, x_\varepsilon)}{1+d_{\mathcal{H}}(x,
x_\varepsilon)}\right)}
\left( 1+\abs{x} \right)^{-(N-2)}, \hspace{0.5cm} x \in \R^N. 
\]
 \end{enumerate}
Moreover,
\begin{enumerate}[(1)]
\item \label{Wepsiloninfinityquadratic} If $\liminf_{\abs{x}\rightarrow\infty} V(x)\abs{x}^{2} > 0$, then there exists
$\lambda > 0$, $\nu > 0$ and $C > 0$ such that for $\varepsilon > 0$ small enough, 
\[
  W_\varepsilon (x) \leq C \exp{\left( -\frac{\lambda}{\varepsilon} \frac{d_{\mathcal{H}}(x, x_\varepsilon)}{1+d_{\mathcal{H}}(x,
x_\varepsilon)}\right)}
\left( 1+\abs{x} \right)^{-\frac{\nu}{\varepsilon}}.
\]
\item \label{WepsiloninfinitySuperquadratic} If $\liminf_{\abs{x}\rightarrow\infty} V(x)\abs{x}^{\alpha} > 0$ with
$\alpha > 2$, then there exists $\lambda > 0$ and $C > 0$ such that for $\varepsilon > 0$ small enough, 
\[
W_\varepsilon (x) \leq C\exp \left(-\frac{\lambda}{\varepsilon}\frac{d_{\mathcal{H}}(x, x_\varepsilon)}{1+d_\mathcal{H}(x,
x_\varepsilon)}(1+\abs{x})^\frac{2-\alpha}{2}\right).
\]
\item \label{Wepsilon0quadratic} If $\liminf_{\abs{x}\rightarrow 0} V(x)\abs{x}^{2} > 0$, then there exists $\lambda >
0$, $\nu > 0$ and $C > 0$ such that for $\varepsilon > 0$ small enough, 
\[
  W_\varepsilon (x) \leq C \exp{\left( -\frac{\lambda}{\varepsilon} \frac{d_{\mathcal{H}}(x, x_\varepsilon)}{1+d_{\mathcal{H}}(x,
x_\varepsilon)}\right)}
\left( \frac{\abs{x}}{1+\abs{x}}  \right)^{\frac{\nu}{\varepsilon}}.
\]
\item \label{Wepsilon0Superquadratic}If $\liminf_{\abs{x}\rightarrow 0} V(x)\abs{x}^{\alpha} > 0$ with $\alpha >
2$, 
then there exists $\lambda > 0$ and $C > 0$ such that for $\varepsilon > 0$ small enough, 
\[
  W_\varepsilon (x) \leq C\exp \left(-\frac{\lambda}{\varepsilon}\frac{d_{\mathcal{H}}(x, x_\varepsilon)}{1+d_\mathcal{H}(x,
x_\varepsilon)}\left(\frac{\abs{x}}{1+\abs{x}}\right)^\frac{\alpha-2}{2}\right).
\]
\end{enumerate}
\end{lemma}
 \begin{proof}
Let $\Psi_\varepsilon$ be given by Lemma \ref{lemmaPsiepsilon}. Choose a set $U \subset \R^N$ such that $\Bar{\Lambda}
\subset U$, $0 \not \in \Bar{U}$ and $\Bar{U}$ is compact. 
Choose $\tilde{\Psi}_\varepsilon \in C^2(\R^N\setminus\lbrace 0 \rbrace) \cap H^1_{\text{loc}}(\R^N)$ such that
$\tilde{\Psi}_\varepsilon = \Psi_\varepsilon$ in $\R^N\setminus U$ and $\Tilde{\Psi}_\varepsilon = 1$ in
$\Lambda$. In view of the estimate of Lemma \ref{lemmaPsiepsilon}, one can also ensure that  
$\sup_{\varepsilon > 0} \norm{\tilde{\Psi}_\varepsilon}_{L^\infty(U)}< \infty$.  
Choose $R>0$ such that 
\begin{equation}
\label{conditionRb}
 R < \liminf_{\varepsilon \to 0} \dist(x_\varepsilon, \partial \Lambda).
\end{equation}
Let $\Phi^{x_\varepsilon}_{\varepsilon}$ be given by \eqref{def:Phieps} and set
\begin{align*}
 w_{\varepsilon}(x) := \left\{ \begin{array}{ll}  
  \Phi^{x_\varepsilon}_{\varepsilon}(x) & \textnormal{if}\ x \in  B_\mathcal{H}(x_{\varepsilon},R), \\
  \Tilde{\Psi}_\varepsilon(x) & \textnormal{if}\ x \in  \R^N\setminus B_\mathcal{H}(x_{\varepsilon},R).
 \end{array} 
 \right.
\end{align*}
By \eqref{conditionRb}, for $\varepsilon$ small enough, $ B_\mathcal{H}(x_{\varepsilon},R) \subset \Lambda$ so that
$w_{\varepsilon} \in
C^{1,1}(\R^N)$. Moreover, if $\varepsilon$ is small enough, Lemma~\ref{CompPeak} is applicable and in $
B_\mathcal{H}(x_{\varepsilon},R) \setminus  B_\mathcal{H}(x_{\varepsilon},\varepsilon \rho)$, we have
\begin{align*}
 -\varepsilon^2 \left( \Delta + H\right)  w_{\varepsilon} + (1-\mu) V w_{\varepsilon} \geq -\varepsilon^2 \Delta
\Phi^{x_\varepsilon}_{\varepsilon} +
(1-\mu) V \Phi^{x_\varepsilon}_{\varepsilon} \geq 0.
\end{align*}
In $\Lambda\setminus B_\mathcal{H}(x_{\varepsilon},R)$, one has
\begin{align*}
 -\varepsilon^2 \left( \Delta + H\right)  w_{\varepsilon} + (1-\mu) V w_{\varepsilon} = -\varepsilon^2  H + (1-\mu) \left(
\inf_{\Lambda} V \right) \geq 0,
\end{align*}
for $\varepsilon$ small enough.
In $U \setminus \Lambda$, one has
\begin{align*}
 -\varepsilon^2 \left( \Delta + H\right)  w_{\varepsilon} + (1-\mu) V w_{\varepsilon} = -\varepsilon^2  \left( \Delta + H\right)
\tilde{\Psi}_\varepsilon  + (1-\mu) V \tilde{\Psi}_\varepsilon \geq 0,
\end{align*}
for $\varepsilon$ small enough since $V\Tilde{\Psi}_\varepsilon$ is positive on $\overline{U}$.
Finally, in $\R^N\setminus U$, one has
\[
 -\varepsilon^2 \left( \Delta + H\right)  w_{\varepsilon} + (1-\mu) V w_{\varepsilon} = -\varepsilon^2  \left( \Delta +
H\right)\Psi_\varepsilon   + (1-\mu) V \Psi_\varepsilon = 0. 
\]
We set
\begin{align*}
 W_{\varepsilon}(x) := \frac{w_{\varepsilon}(x)}{\cosh \lambda \left( \frac{R}{\varepsilon} - \rho\right)},
\end{align*}
where $\lambda$ is chosen as in the previous lemma. It is standard to see that $W_{\varepsilon}$ satisfies properties
\eqref{barrierb} and \eqref{barrierc}. Statement \eqref{barrierd} follows from Lemma~\ref{lemmaPsiepsilon}.
The other conclusions follow from Lemma~\ref{lemmaImprovedDecay}.
\end{proof}

Thanks to the previous lemma, we obtain an upper bound on the solutions $(u_{\varepsilon})_{\varepsilon>0}$ of
\eqref{p2}. 

\begin{proposition}\label{Prop:decay}
 Suppose that the assumptions of Proposition \ref{Prop1} are satisfied. Let $(u_{\varepsilon})_{\varepsilon>0} \subset
H^1_{V,\mathcal{H}}(\R^N)$ be the positive solutions of \eqref{p2} found in Theorem \ref{Th1} and
$(x_{\varepsilon})_{\varepsilon>0} \subset \Lambda$ be such that
\begin{align*}
 \liminf_{\varepsilon\to 0} u_{\varepsilon}(x_{\varepsilon}) > 0.
\end{align*}
Then there exist $C>0$, $\lambda > 0$ and $\varepsilon_0>0$ such that for all $\varepsilon \in (0,\varepsilon_0)$,
\begin{align}\label{decay}
 u_{\varepsilon}(x) \leq C \exp{\left( -\frac{\lambda}{\varepsilon} \frac{d(x,S^k_{\varepsilon})}{1+d(x,S^k_{\varepsilon})}\right)}
( 1+\abs{x} )^{-(N-2)}, \hspace{0.5cm} x \in \R^N.
\end{align}
Moreover, \eqref{Wepsiloninfinityquadratic}, \eqref{WepsiloninfinitySuperquadratic}, \eqref{Wepsilon0quadratic} and
\eqref{Wepsilon0Superquadratic} in Lemma~\ref{lemma:barrier} hold with $u_\varepsilon$ in place of $W_\varepsilon$.
\end{proposition}
\begin{proof}
By Lemma \ref{lemma:subsol}, there exist $\rho > 0$ and $\varepsilon_0>0$ such that for all $\varepsilon \in (0,\varepsilon_0)$,
the
solution $u_{\varepsilon}$ satisfies inequation \eqref{lowsol}. Further, $\norm{u_{\varepsilon}}_{L^\infty
(B_\mathcal{H}(x_{\varepsilon},\varepsilon \rho))}$ is bounded as $\varepsilon \to 0$ in view of
Lemma~\ref{convergentSubsequence}.
Let
$(W_{\varepsilon})_{\varepsilon}$ be the family of barrier functions given by Lemma \ref{lemma:barrier}. By
Proposition~\ref{Th:comp}, we have
\begin{align*}
 u_{\varepsilon}(x) \leq \norm{u_{\varepsilon}}_{L^\infty (B_\mathcal{H}(x_{\varepsilon},\varepsilon \rho))} W_{\varepsilon}(x)
\hspace{0.5cm} \text{in}\ \R^N\setminus B_\mathcal{H} (x_{\varepsilon},\varepsilon \rho),
\end{align*}
and the conclusion comes from Lemma~\ref{lemma:barrier}.
\end{proof}

\bigbreak

We are now in a position to prove Theorem \ref{Th:main}.

\begin{proof}[Proof of Theorem \ref{Th:main}]
We know from Theorem \ref{Th1} that the modified equation \eqref{p2} possesses a positive solution $u_{\varepsilon} \in
H^1_{V,\mathcal{H}}(\R^N)$. In order to prove that for $\varepsilon$ small enough, this solution actually solves
\eqref{p1}, it suffices to show that, for every $x \in (\R^N \setminus \{0\})\setminus \Lambda$, one
has
\[
 K(x) \frac{f(u_{\varepsilon}(x))}{u_{\varepsilon}(x)} \leq \varepsilon^2 H(x) + \mu V(x).
\]

Assume that $V$ and $K$ satisfy $(\mathcal{G}_{\infty}^1)$ and $(\mathcal{G}_{0}^1)$, by Proposition \ref{Prop:decay}
and
assumptions $(f_4)$ and $(f_1)$, if $\varepsilon > 0$ is small enough, we have for all $x \in \R^N\setminus \Lambda$,
\begin{align*}
  K(x) \frac{f(u_{\varepsilon}(x))}{u_{\varepsilon}(x)} &\leq K(x) \frac{f\left(C
e^{-\frac{\lambda}{\varepsilon}} (1+\abs{x})^{-(N-2)}\right)}{C e^{-\frac{\lambda}{\varepsilon}} (1+\abs{x})^{-(N-2)}} \\ 
&\leq C e^{-\frac{\lambda}{\varepsilon}(q-1)} (1+\abs{x})^{\sigma-(N-2)(q-1)} \\ 
&\leq \frac{\varepsilon^2 \kappa}{\abs{x}^2 \left( (\log \abs{x})^2+1 \right)^{\frac{1+\beta}{2}}} = \varepsilon^2 H(x).
\end{align*}
The other cases can be treated in a similar way.
\end{proof}

In some settings, it is interesting to determine whether the solutions are in $L^2$. We obtain as a byproduct the following

\begin{corollary}\label{cor:L2}
 Let $u_{\varepsilon}$ be the solution of \eqref{p1} found in Theorem \ref{Th:main}. If $N \geq 5$ or $\liminf_{\abs{x} \to \infty} \abs{x}^2V(x) > 0$, then, for $\varepsilon$
small enough, $u_{\varepsilon} \in L^2(\R^N)$.
\end{corollary}
\begin{proof}
This follows immediately from Proposition~\ref{Prop:decay}.
\end{proof}

\section{The two-dimensional case}\label{rem:N=2}
 In dimension $N = 2$, the method has to be modified because the classical Hardy inequality fails on unbounded domains
of $\R^2$. Let us recall the Hardy-type inequality that was proved in \cite{MVS}*{Lemma 6.1}:
\begin{lemma}
\label{lemmaHardyMV}
Let $R > r$. Then there exists $C > 0$ such that for every $u \in \mathcal{D}(\R^2)$,
\[
 \int_{\R^2} \abs{\nabla u}^2 + C \int_{B(0,R)\setminus B(0,r)} u^2 \geq \frac{1}{4}
\int_{\R^2\setminus B(0, R)} \frac{u^2(x)}{\abs{x}^2 \left(\log\frac{\abs{x}}{r}\right)^2}\: dx.
\]
\end{lemma}

We deduce therefrom
\begin{lemma}
\label{lemmaNewHardy}
If $V \in C(\R^2\setminus \{0\})$ is nonnnegative and non identically $0$, then there exists $\kappa_0 > 0$ such that for $\varepsilon > 0$ sufficiently small, for every $u \in \mathcal{D}(\R^2)$, 
\[
 \kappa_0 \int_{\R^2} \frac{u^2(x)}{\abs{x}^2 \Bigl(1+(\log \abs{x})^2\Bigr)}\:dx\leq 
\int_{\R^2} \varepsilon^2\abs{\nabla u}^2 + V u^2.
\]
\end{lemma}
\begin{proof}
One sees that by the conformal transformation $x \mapsto \frac{x}{\abs{x}^2}$, Lemma~\ref{lemmaHardyMV} becomes 
\[
 \int_{\R^2} \abs{\nabla u}^2 + C \int_{B(0,R)\setminus B(0,r)} u^2 \geq \frac{1}{4}
\int_{B(0, r)} \frac{u^2(x)}{\abs{x}^2 \left(\log \abs{x}\right)^2}\: dx.
\]
Therefore, one has
\[
 \int_{\R^2} \frac{u^2(x)}{\abs{x}^2 \Bigl(1+(\log \abs{x})^2\Bigr)}\:dx\leq 
C\Bigl(\int_{\R^2} \abs{\nabla u}^2+\int_{B(0, 2)\setminus B(0, 1/2)} u^2\Bigr).
\]

Since $V$ is continuous and does not vanish identically, there exists $\Bar{x} \in \R^2$ and $\Bar{r} > 0$, 
such that $\inf_{B(\Bar{x},\Bar{r})} V > 0$. Hence, there exists $C > 0$ such that 
\[
 \int_{B(0,2)\setminus B(0,1/2)} u^2 \leq C\Bigl(\int_{\R^2} \abs{\nabla u}^2 + V \abs{u}^2\Bigr).
\]

Bringing the inequalities together, there exists $C > 0$ such that 
\[
 \int_{\R^2} \frac{u^2(x)}{\abs{x}^2 \bigl(1+(\log \abs{x})^2\bigr)}\:dx\leq 
C\int_{\R^2} (\abs{\nabla u}^2+V u^2).
\]
This brings the conclusion when $\varepsilon > 0$ is small enough.
%
%
%
\end{proof}

The space $H^1_V(\R^2)$ can thus be defined as in the case $N>2$ as the closure of $\mathcal{D}(\R^2)$ with respect to the
norm defined by \eqref{eqNorm}.

The penalization potential $H : \R^2 \to \R$ is defined by
\begin{align*}
 H(x) := \frac{\kappa}{\abs{x}^2 \left( 1+(\log \abs{x})^2 \right)^{\frac{2+\beta}{2}}}.
\end{align*}
where $\beta > 0$ and $\kappa \in (0, \kappa_0)$
We see that
\[
 H(x) \leq \frac{\kappa}{\abs{x}^2 \left( 1+(\log \abs{x})^2 \right)}.
\]
Together with Lemma~\eqref{lemmaNewHardy}, this ensures positivity of the quadratic form associated to $-\varepsilon^2
(\Delta+H) + V$. 

As in the case $N>2$, this inequality implies the following comparison principle.
\begin{proposition}\label{Th:comp2}
Let $\Omega \subset \R^2$ be a smooth domain. Let $v,w \in H^1_{\textnormal{loc}}(\Omega)\cap C(\overline{\Omega})$ be such
that $\nabla (w-v)_- \in L^2(\Omega)$, $(w-v)_-/(\abs{x}(1+\lvert\log \abs{x}\rvert)) \in L^2(\Omega)$ and
\[
 - \varepsilon^2 (\Delta + H) w + V w \geq - \varepsilon^2 (\Delta + H) v + V v, \hspace{1cm}  \text{ in } \Omega.
\]
If $\partial\Omega \neq \emptyset$, assume also that $w \geq v$ on $\partial\Omega$. Then $w \geq v$ in $\Omega$.
\end{proposition}

One continues the proof of Theorem~\ref{Th:main} as in the case $N \geq 3$. 
In Proposition~\ref{prop:PS}, one takes $A_{\lambda} := \overline{B(0, e^{e^\lambda)}} 
\setminus B(0, e^{-e^{\lambda}})$ and
 \begin{align*}
  \eta_{\lambda}(x) := \zeta\left( \frac{\log \abs{\log \abs{x}}}{\lambda} \right).
 \end{align*}
One then obtains estimate \eqref{ps5} by using Lemma~\ref{lemmaNewHardy} instead of Hardy's inequality.
The only other notable difference lies in the choice of the function $W$ in the proof of Lemma~\ref{lemmaPsiepsilon}, 
where one follows the construction of \cite{MVS}*{Lemma 6.3}, i.e.
\[
  W(x)=\beta (\beta+1)-\kappa\abs{\log \abs{x}}^{-\beta}.
\]


\end{document}